\documentclass{amsart}

\usepackage[T1]{fontenc}
\usepackage{amsmath}
\usepackage{amssymb}
\usepackage{amsthm}
\usepackage{epsfig}
\usepackage{subcaption}
\usepackage{xcolor}
\usepackage{hyperref}
\input xy
\xyoption{all}
\usepackage[all]{xy} 

\newtheorem{theorem}{Theorem}[section]
\newtheorem{corollary}[theorem]{Corollary}
\newtheorem{proposition}[theorem]{Proposition}
\theoremstyle{definition}
\newtheorem{definition}[theorem]{Definition}
\newtheorem{example}[theorem]{Example}
\newtheorem{remark}[theorem]{Remark}
\numberwithin{equation}{section}

\newcommand{\R}{\mathcal{R}}
\newcommand{\Ra}[2]{\R_{#1,#2}}
\newcommand{\C}{\mathcal{C}}
\newcommand{\Ca}[2]{\C_{#1,#2}}
\newcommand{\M}{\mathcal{M}}
\newcommand{\Ma}[2]{\M_{#1,#2}}
\newcommand{\ma}[1]{\mathcal{O}_{#1}}

\newcommand{\Tma}{\tilde{\mathcal{O}}}
\newcommand{\K}{\mathcal{K}}
\newcommand{\Ka}[2]{\K_{#1,#2}}
\newcommand{\Pc}{\mathcal{P}}
\newcommand{\Sc}[1]{\mathcal{S}_{#1}}
\newcommand{\T}{\mathcal{T}}
\newcommand{\Ta}[2]{\T_{#1,#2}}

\newcommand{\OO}{\overline{0}}

\DeclareMathOperator{\diag}{{\it diag}}
\DeclareMathOperator{\h}{{\it horiz}}
\DeclareMathOperator{\di}{{div}}

\begin{document}

\title[Knight's Tours on Non-Orientable Surfaces]
{Nullhomotopic and Generating Knight's Tours on Non-Orientable Surfaces}

\author{Bradley Forrest}
\address{101 Vera King Farris Drive; Galloway, NJ 08205}
\email{bradley.forrest@stockton.edu}

\author{Zachary Lague}
\email{zml2020@nyu.edu}



\begin{abstract}
We investigate closed knight's tours on M\"{o}bius strip and Klein bottle chess boards.  In particular, we characterize the board dimensions that admit tours that are nullhomotopic and the board dimensions that admit tours that realize generators of the fundamental groups of each of the surfaces.
\end{abstract}

\maketitle

\section{\bf Introduction}
The knight's tour problem is an ancient and classic graph theoretic puzzle inspired by the knight in chess.  Chess is traditionally played on a square board made up of 64 smaller squares laid out in an $8 \times 8$ pattern. The knight moves in an ``L'' shape, moving 2 squares in one of the four grid directions and 1 square in a perpendicular direction.  The knight's tour puzzle challenges the solver to move the knight to each of the 64 squares without repeating any square.  Historically, a full solution is required to be a \textit{closed} tour, a tour that ends with the knight one move away from its starting square.  

The traditional knight's tour problem is well-studied, as are many of its generalizations \cite{yang-zhu-jiang-huang}, \cite{chia-ong}, \cite{demaio}, \cite{demaio-mathew}, \cite{erde-golenia-golenia}. Most relevantly, several authors have characterized the dimensions of various surfaces that admit knight's tours.  In particular, Schwenk characterized the dimensions of rectangular boards that admit a knight's tour \cite{schwenk}. Watkins and Hoenigman extended this result and characterized the dimensions of tori that admit knight's tours \cite{watkins-hoenigman}, and Watkins later established the analogous results for cylinders, M\"{o}bius strips, and Klein bottles \cite{watkins}.  Narrowing the collection of allowed tours, Forrest and Teehan characterized the dimensions of cylinders and tori that admit nullhomotopic tours and the dimensions that admit tours that are homotopic to a generator of the fundamental group of the surface \cite{forrest-teehan}.  In this work, we extend on the results of Forrest and Teehan and characterize the analogous dimensions of M\"{o}bius strips and Klein bottles.  

More concretely, we study knight's tours on $m \times n$ M\"{o}bius strips $M$ and $m \times n$ Klein bottles $K$.  These are both quotients of the rectangle with width $m$ and height $n$, which we will consider to be tiled by $mn$ squares with unit side length. Specifically, $M$ is obtained by twisting the top edge of the rectangle and identifying it with the bottom edge.  This is the M\"{o}bius strip of width $m$ whose boundary circle has length $2n$.  Similarly, $K$ is the quotient of $M$ obtained by identifying the left and right edges of the rectangle without twisting. For each of these surfaces, we construct a graph with a vertex for each unit square in the tiling of the surface and an edge between two squares if the knight can move from one square to the other.  For most board dimensions, this is a graph; however, we include an edge for each distinct move that the knight could make between the two squares, so for small values of $m$ and $n$, this construction can produce a pseudograph or multigraph.  These graphs map naturally into $M$ and $K$.  

By choosing a basepoint vertex $v$ in these graphs, directed knight's tours beginning at $v$ determine elements of $\pi_1(M,v)$ and $\pi_1(K,v)$.  In this work, we characterize the values of $m$ and $n$ that admit knight's tours of the following types:
\begin{itemize}
\item Knight's tours that realize the identity in $\pi_1(M,v)$ or $\pi_1(K,v)$.  We call these {\it nullhomotopic tours}.
\item Knight's tours that realize a generator in $\pi_1(M,v)$. We call these {\it generating tours}.
\item Knight's tours that realize an element of $\pi_1(K,v)$ that is the image of a generator of fundamental group of the cylinder under the homomorphism induced by the standard quotient map. We call these {\it cylindrical tours}.  
\item Knight's tours that realize an element of $\pi_1(K,v)$ that is the image of a generator of fundamental group of the M\"{o}bius strip under the homomorphism induced by the standard quotient map. We call these {\it M\"{o}bius tours}.
\end{itemize}

These characterizations are given in Theorems \ref{thm:MSNH}, \ref{thm:MSGen}, \ref{thm:KBNull}, \ref{thm:KBCyl}, and \ref{thm:KBMob} respectively.

This paper is organized as follows:  In Section 2, we define the surfaces and graphs that we use throughout this work and review foundational graph theoretic results.  Section 3 introduces covering space theory and applies it to establish an important technical result that enables widening knight's tours on M\"{o}bius strips.  Sections 4 and 5 classify the dimensions of M\"{o}bius strip boards that admit nullhomotopic and generating tours respectively. Section 6 extends these results to Klein bottle boards by classifying the dimensions of Klein bottles that admit nullhomotopic tours, the dimensions that support cylindrical tours, and the dimensions that support M\"{o}bius tours.  Section 7 discusses potential future

\section{\bf Background on Graphs and Surfaces}

In this section, we will define the surfaces, graphs, and maps that are relevant to our work.  This includes reviewing results of Schwenk, Watkins, and Ralston.  We follow the discussion presented in \cite[Section 2]{forrest-teehan} closely, though with some details suppressed for brevity; see that work for a more thorough discussion.

\subsection{Surfaces}

We will use four different compact surfaces in this work: rectangles, cylinders, M\"{o}bius strips, and Klein bottles, each with dimensions $m \times n$. These surfaces are all quotients of the rectangle which we coordinatize by the set of ordered pairs of real numbers $(x,y)$ where $x \in [0,m]$ and $y \in [0,n]$.  The cylinder $C$ is the quotient obtained by identifying $(0,y)$ and $(m,y)$ for each $y \in [0,n]$, while the M\"{o}bius strip $M$ is given by identifying $(x,0)$ and $(m-x,n)$ for each $x \in [0,m]$.  We obtain the Klein bottle $K$ by performing both of these identifications.  Additionally, we will use two non-compact surfaces; let $P$ be the Cartesian plane and $S$ be the subset of the Cartesian plane between the lines $x=0$ and $x=m$. Note that $S$ is the universal cover of $M$ and $P$ is the universal cover of $K$.  The covering maps $p_M \colon S \to M$ and $p_K \colon P \to K$ are both given by $(x,y) \mapsto ((-1)^{y \di n}x \mod m,y \mod n)$.  

Each of these six surfaces can be tiled by unit squares. We take the compact surfaces to be tiled by $mn$ unit squares and the non-compact surfaces to be tiled by infinitely many unit squares. The sides of each square are parallel to the $x$ or $y$ axes and the center of each square has the form $(a+.5,b+.5)$ for some integers $a$ and $b$.

\subsection{Graphs}

We will use graphs that model the moves a knight could make on each of the six surfaces described in the previous subsection.  Each graph corresponding to a compact surface has a vertex for each ordered pair of integers $(a,b)$ with $0 \leq a < m$ and $0 \leq b < n$ and an edge for each distinct move a knight could make from one vertex to another.  To be more precise, on a rectangle there is a knight move from $(a,b)$ to $(a+c,b+d)$ so long as $\{|c|,|d|\} = \{1,2\}$.  If $c_0$ and $d_0$ are integers so that $\{|c_0|,|d_0|\} = \{1,2\}$, then we regard the knight moves from $(a,b)$ to $(a+c_0,b+d_0)$ and from $(a+c_0,b+d_0)$ to $(a+c_0-c_0,b+d_0-d_0) = (a,b)$ as a single edge. The graph modeling the possible moves of a knight on a rectangle, $\Ra{m}{n}$, is the graph whose edges are precisely these equivalence classes of knight moves.  To obtain the other graphs, we add edges to $\Ra{m}{n}$ that represent each quotient of the rectangle. For $\Ca{m}{n}$, we add edges from $(a,b)$ when $a+c \ne (a+c \mod m)$; specifically, these are edges from $(a,b)$ to $((a+c) \mod m, b+d)$. Note that $\Ca{m}{n}$ is $\Ca{n}{m}$ as discussed in \cite{forrest-teehan}.  For $\Ma{m}{n}$, we add edges for knight moves from $(a,b)$ when $b+d \ne (b+d \mod n)$, and these edges are from $(a,b)$ to $(m-a-c,(b+d) \mod n)$. Lastly, for $\Ka{m}{n}$, we include every edge in at least one of $\Ca{m}{n}$ and $\Ma{m}{n}$ and add edges for knight moves from $(a,b)$ when $a+c \ne (a + c \mod m)$ and $b+d \ne (b+d \mod n)$.  These edges are from $(a,b)$ to $((m - a - c) \mod m, (b + d) \mod n)$.

These definitions cause technical issues for boards with small values of $m$ and $n$, as these graphs can sometimes be multigraphs or pseudographs.  For example, $\Ka{1}{1}$ is a pseudograph with one vertex and four edges.  Similarly, $\Ca{1}{2}$ is a multigraph, as it has two vertices between which there are two edges.  

We will also make use of graphs modeling knight moves on the infinite strip and plane.  Specifically, $\Sc{m}$ is the graph that has a vertex for each ordered pair of integers $(a,b)$ with $0 \leq a < m$, and $\Pc$ has a vertex for each ordered pair of integers.  The edges in these graphs are defined analogously to the edges in $\Ra{m}{n}$. There are combinatorial maps $\phi_M \colon \Sc{m} \to \Ma{m}{n}$ and $\phi_K \colon \Pc \to \Ka{m}{n}$ given by mapping the vertex $(a,b)$ to $((-1)^{b \di n}a \mod m, b \mod n)$. These maps take the edge from $(a,b)$ to $(a+c,b+d)$ to the edge from $((-1)^{b \di n}a \mod m, b \mod n)$ to $(((-1)^{b \di n}(a+c)) \mod m, (b+d) \mod n)$. In addition to preserving the graph structure, $\phi_M$ and $\phi_K$ are covering maps. 

\subsection{Mapping Graphs to Surfaces}

There are maps from the graphs $\Ma{m}{n}$, $\Ka{m}{n}$, $\Sc{m}$, and $\Pc$ to $M$, $K$, $S$, and $P$ respectively. These maps send each vertex $(a,b)$ to the point $(a + .5, b + .5)$ and each edge continuously to the geodesic line segment connecting the vertices.  We denote these maps $i_M \colon \Ma{m}{n} \to M$, $i_K \colon \Ka{m}{n} \to K$, $i_S \colon \Sc{m} \to S$, and $i_P \colon \Pc \to P$.  For $i_M$ and $i_K$, some care is needed in mapping the edges, as $\Ma{m}{n}$ and $\Ka{m}{n}$ may be multigraphs or pseudographs, or equivalently, the geodesic line segment between two vertices in $M$ and $K$ may not be unique; this technical issue does not exist for $i_S$ and $i_P$, as $\Sc{m}$ and $\Pc$ are simple graphs.  Let $e$ be an edge in $\Ma{m}{n}$.  We define $i_M(e) = p_M \circ i_S(e')$, where $e'$ is any lift of $e$ in $\Sc{m}$.  Likewise, if $e$ is an edge in $\Ka{m}{n}$, we define $i_K(e) = p_K \circ i_P(e')$, where $e'$ is any lift of $e$ in $\Pc$.  That is, we define $i_M$ and $i_K$ so that the diagram below commutes. 
\begin{figure}[h]
\[
    \xymatrix{
      \Sc{m} \ar[r]^{\textstyle i_S}\ar[d]_{\textstyle\phi_M} & S\ar[d]^{\textstyle p_M}  & \qquad & \Pc \ar[r]^{\textstyle i_P}\ar[d]_{\textstyle\phi_K} & P\ar[d]^{\textstyle p_K}  &\\
    \Ma{m}{n} \ar[r]_{\textstyle i_M} & M & \qquad & \Ka{m}{n} \ar[r]_{\textstyle i_K} & K 
    }
\]
\end{figure}

Since $\Ra{m}{n}$, $\Ca{m}{n}$, $\Ma{m}{n}$, and $\Ka{m}{n}$ have all been defined to model the moves a knight can make on a surface, we refer to Hamiltonian cycles on these graphs as knight's tours and Hamiltonian paths that are not cycles as open knight's tours.

\subsection{Graph Theoretic Results}

In this subsection, we review some classical results on knight's tours that are pertinent to our work. One such result is Schwenk's characterization of the dimensions of rectangular boards that admit closed knight's tours:

\begin{theorem}[Schwenk \cite{schwenk}]
The graph $\Ra{m}{n}$ with $m \leq n$ and $n > 1$ admits a knight's tour if and only if none of the following are true:
\begin{itemize}
\item $m$ and $n$ are simultaneously odd;
\item $m=1$, $2$, or $4$; or
\item $m = 3$ while $n=4$, $6$, or $8$. 
\end{itemize}
\label{thm:schwenk} 
\end{theorem}

Since $\Ma{m}{n}$ and $\Ka{m}{n}$ canonically contain $\Ra{m}{n}$, any closed knight's tour on $\Ra{m}{n}$ is a nullhomotopic tour on $\Ma{m}{n}$ and $\Ka{m}{n}$; therefore, Theorem \ref{thm:schwenk} is a sufficient condition establishing the existence of nullhomotopic tours on $\Ma{m}{n}$ and $\Ka{m}{n}$.

Watkins and Hoenigman completed the analogous characterizations for tori in \cite{watkins-hoenigman}, while Watkins extended this to cylinders, M\"obius strips, and Klein bottles in \cite{watkins}.  In our work, Watkins' results regarding M\"{o}bius strips are of particular utility:

\begin{theorem}[Watkins \cite{watkins}; pg. 82]
The multigraph $\Ma{m}{n}$ admits a knight's tour if and only if none of the following are true:
\begin{itemize}
\item $m = 1$ and $n > 1$;
\item $m = 2$ and $n$ is even;
\item $m = 3$ and $n = 1$ or $4$;
\item $m = 4$ and $n$ is odd; or
\item $m = 5$ and $n = 1$.
\end{itemize}
\label{thm:watkinsms}
\end{theorem}

Theorem \ref{thm:watkinsms} gives a necessary condition for the existence of nullhomotopic and generating tours on M\"{o}bius strips.  We will not make use of the analogous result for Klein bottles because Watkins proved that $\Ka{m}{n}$ admits a knight's tour for all $m$ and $n$ \cite[pg. 81]{watkins}.

Another useful result regarding the existence of generating tours on $\Ma{m}{n}$ is due to Ralston:

\begin{theorem}[Ralston \cite{ralston}]
Let $m$ and $n$ be odd with both $m$ and $n$ greater than or equal 5 and at least one them strictly greater than 5. If $(a,b)$ and $(c,d)$ are a distinct pair of vertices where $a+b$ and $c+d$ are even, then there is an open knight's tour in $\Ra{m}{n}$ starting at $(a,b)$ and ending at $(c,d)$.
\label{thm:ralston}
\end{theorem}

An open knight's tour on $\Ra{m}{n}$ starting and ending at a pair of well-chosen vertices produces a generating tour on $\Ma{m}{n}$, so in some cases Theorem \ref{thm:ralston} suffices to demonstrate that generating tours exist on $\Ma{m}{n}$.

\section{{\bf Applying Covering Space Theory to Knight's Tours}}

Covering space theory provides a simple way to determine the topology of knight's tours in $\Ma{m}{n}$ and $\Ka{m}{n}$.  In particular, the end points of the lifts of the tours to $\Sc{m}$ or $\Pc$ determine the topology of the tour.  In this section, we make this relationship precise and examine some consequences that we will find useful.

\subsection{Covering Space Theory}

This subsection reviews the necessary results from covering space theory.  We have suppressed some of the details, as we closely follow the discussion in \cite[Sections 3]{forrest-teehan}; please see that work for a more thorough exposition.

Knight's tours can be regarded as loops in the graphs and topological spaces defined in the previous section, and for a loop $f \colon (I,0) \to (X,x)$, where $I$ is the closed unit interval and $X$ is a topological space, we let $[f]$ denote its homotopy class.  We will take $(0,0) = \overline{0}$ as our basepoint in each graph and $v = (.5,.5)$ as our basepoint in each surface.  Each knight's tour $f \colon (I,0) \to (\Ma{m}{n}, \overline{0})$ or $(\Ka{m}{n}, \overline{0})$ gives an element of $\pi_1(M,v)$ and $\pi_1(K,v)$ respectively; these are $[i_M \circ f]$ and $[i_K \circ f]$.  Note that {\it nullhomotopic tours} in $\Ma{m}{n}$ are precisely tours corresponding to the identity elements of $\pi_1(M,v)$; analogously, nullhomotopic tours in $\Ka{m}{n}$ correspond to the identity in $\pi(K,v)$. The tours in $\Ma{m}{n}$ that give the generator of $\pi_1(M,v)$ are the {\it generating tours}.  Let $\h, \diag \colon I \to P$ be the straight line paths from $(.5,.5)$ to $(m+.5,.5)$ and from $(.5,.5)$ to $(m-.5,n+.5)$ respectively. Then a knight's tour in $\Ka{m}{n}$ whose image under $i_M$ is homotopic to $p_M \circ \h$ is a {\it cylindrical tour}; if instead the image is homotopic to $p_M \circ \diag$, the tour is a {\it M\"{o}bius tour}. 

The following classical theorem from covering space theory is our main tool to determine whether a knight's tour is any of the types that we study in this work:

\begin{theorem}
Let $\tilde{X}$ be the universal cover of $X$ with covering map $p \colon \tilde{X} \to X$, and let $x \in X$ with $\tilde{x} \in p^{-1}(x)$.  Let $f, g \colon (I, 0) \to (X, x)$ be two loops with lifts $\tilde{f}, \tilde{g} \colon (I, 0) \to (\tilde{X}, \tilde{x})$.  Then $f$ and $g$ are path homotopic if and only if $\tilde{f}(1) = \tilde{g}(1)$.
\label{thm:cover}
\end{theorem}

Now let $f \colon (I, 0) \to (\Ma{m}{n},\OO)$ be a loop and note that $f$ lifts to a path $\tilde{f}$ in $\Sc{m}$.  By Theorem \ref{thm:cover}, the endpoints of the lift determine the homotopy class of $i_M \circ f$.  Likewise, when $g \colon (I, 0) \to (\Ka{m}{n},\OO)$ is a loop, $g$ lifts to a path $\tilde{g}$ in $\Pc$ whose endpoints determine the homotopy class of $i_K \circ g$.

\begin{corollary}
Let $f \colon (I, 0) \to (\Ma{m}{n},\OO)$ and $g \colon (I, 0) \to (\Ka{m}{n},\OO)$ be knight's tours with lifts $\tilde{f} \colon I \to \Sc{m}$ and $\tilde{g} \colon I \to \Pc$ with given initial points $\tilde{f}(0) = \OO$ and $\tilde{g}(0) = \OO$.  Then:
\begin{enumerate}
\item  The loop $f$ is a nullhomotopic tour if and only if $\tilde{f}(1) = \OO$.
\item The loop $f$ is a generating tour if and only if $\tilde{f}(1) = (m-1,\pm n)$.
\item The loop $g$ is a nullhomotopic tour if and only if $\tilde{g}(1) = \OO$.
\item The loop $g$ is a cylindrical tour if and only if $\tilde{g}(1) = (\pm m,0)$.
\item  The loop $g$ is a M\"{o}bius tour if and only if $\tilde{g}(1) = (m-1,\pm n)$.
\end{enumerate}
\label{cor:cover}
\end{corollary}

The proof of Corollary \ref{cor:cover} is directly analogous to the proof of Corollary 3.3 in \cite{forrest-teehan}.  Additionally, the definitions of nullhomotopic and generating tours used for $\Ma{m}{n}$ in this work are analogous to those used for cylinders in \cite{forrest-teehan}.  Two results regarding knight's tours on cylinders from \cite{forrest-teehan} will be used in this work and are listed below.  We have adapted these results to our context; recall that $\Ca{m}{n}$ in this work is $\Ca{n}{m}$ in \cite{forrest-teehan}.

\begin{theorem}[\cite{forrest-teehan}, Theorem 4.1]
The multigraph $\Ca{m}{n}$ has a nullhomotopic tour if and only if none of the following are true:
\begin{itemize}
\item $m$ and $n$ are simultaneously odd and at least one of $m$ and $n$ is greater than 1;
\item $n = 1$ and $m > 1$;
\item $n = 2$; or
\item $n = 4$ and $m$ is even.
\end{itemize}
\label{thm:forrest-teehan}
\end{theorem}

\begin{theorem}[\cite{forrest-teehan}, Theorem 5.1]
The multigraph $\Ca{m}{n}$ has a generating tour if and only if none of the following hold:
\begin{itemize}
\item $n = 1, 2, 4$; or
\item $m$ is odd and $n$ is even.
\end{itemize}
\label{thm:forrest-teehan2}
\end{theorem}

The multigraph $\Ca{m}{n}$ is canonically a subgraph of $\Ka{m}{n}$, so Theorems \ref{thm:forrest-teehan} and \ref{thm:forrest-teehan2} provide sufficient conditions for the existence of nullhomotopic and cylindrical tours in $\Ka{m}{n}$ respectively.

\subsection{Colorings and Parity}

Each of $\Ra{m}{n}$, $\Sc{m}$, and $\Pc$ have 2-colorings given by coloring the vertex $(a,b)$ red if $a+b$ is even and blue if $a+b$ is odd; for $\Ra{8}{8}$, this is the coloring of a chess board by light and dark squares. These 2-colorings on $\Sc{m}$ and $\Pc$ are used in the proof of statement (1) of Proposition \ref{prop:oddeven} below.

Note that the 2-coloring on $\Ra{m}{n}$ does not, in general, give a 2-coloring on $\Ma{m}{n}$ nor on $\Ka{m}{n}$, as some of the edges that have been added to create $\Ma{m}{n}$ and $\Ka{m}{n}$ may connect vertices of the same color.  We can give different colorings on $\Sc{m}$ and $\Pc$ by lifting these respective colorings on $\Ma{m}{n}$ and $\Ka{m}{n}$. Note again that these are not 2-colorings on $\Sc{m}$ and $\Pc$, as preimages under $\phi_M$ and $\phi_K$ of edges that are not in $\Ra{m}{n}$ may connect two vertices of the same color.  More formally, these colorings are described by letting each vertex $(a,b)$ in $\Sc{m}$ and $\Pc$ be green if $((-1)^{b \di n}a \mod m) + (b \mod n)$ is even and yellow if that value is odd. These colorings on $\Sc{m}$ and $\Pc$ are used in the proof of statement (2) of Proposition \ref{prop:oddeven} below.  

\begin{proposition}
$\;$
\begin{enumerate}
    \item Let $n$ be odd. If $m$ is odd and at least one of $m$ and $n$ is larger than 1, then there is no nullhomotopic tour on $\Ma{m}{n}$ nor on $\Ka{m}{n}$.
    \item Let $n$ be even.  If $m$ is even, there is no generating tour on $\Ma{m}{n}$ and no M\"obius tour on $\Ka{m}{n}$. If $m$ is odd, there is no cylindrical tour on $\Ka{m}{n}$.  
\end{enumerate} 
\label{prop:oddeven}
\end{proposition}

\begin{proof}
The proof of the statement (1) is analogous to the proofs in Proposition 3.4 of \cite{forrest-teehan}.

For statement (2), let $n$ be even.  Suppose $m$ is even and, for the sake of contradiction, that $f \colon (I,0) \to (\Ma{m}{n}, \OO)$ is a generating tour with lift $\tilde{f}$ and $\tilde{f}(0) = \overline{0}$. Note that, by Corollary \ref{cor:cover}, the starting and ending vertices of $\tilde{f}$ have the same color by the green-yellow coloring discussed above.  In the case when $m$ and $n$ are both even, the green-yellow coloring is not a 2-coloring. The edges from $(a_1,b_1)$ to $(a_2,b_2)$ for which $b_1 \di n \ne b_2 \di n$ are adjacent to two vertices of the same color, while all other edges are adjacent to two vertices of different colors. Therefore, $\tilde{f}$ must traverse an odd number of edges that are adjacent to vertices of the same color. Note that $\tilde{f}$ traverses an even number of edges total and so an odd number of the edges traversed by $\tilde{f}$ are adjacent to two vertices of different colors. Since the starting and ending vertices of $\tilde{f}$ must be of opposite colors, we have a contradiction.  The same proof applies to show that there are no M\"obius tours on $\Ka{m}{n}$.

Now suppose that $m$ is odd and, for the sake of contradiction, suppose that $g \colon (I,0) \to (\Ka{m}{n}, \OO)$ is a cylindrical tour with lift $\tilde{g}$ and $\tilde{g}(0) = \overline{0}$.  Note that, by Corollary \ref{cor:cover}, the starting and ending vertices of $\tilde{g}$ have the same color by the green-yellow coloring discussed above.  In the case when $m$ is odd and $n$ is even, the green-yellow coloring is not a 2-coloring. The edges from $(a_1,b_1)$ to $(a_2,b_2)$ for which $a_1 \di m \ne a_2 \di m$ are adjacent to two vertices of the same color, while all other edges are adjacent to two vertices of different colors. The lift $\tilde{g}$ must traverse an odd number of the edges adjacent to vertices of the same color but an even number of edges total, which leaves an odd number of the traversed edges that are adjacent to vertices of different colors. Therefore, the starting and ending vertices of $\tilde{g}$ have opposite colors, which is a contradiction.
\end{proof}

\subsection{\bf Inclusion and Replacement Paths} 

Theorem \ref{thm:cover} shows that any operations that we can perform on paths that do not alter their endpoints will not change their topology.  In this and the following subsection, we show that this idea allows us to, given a tour on a M\"obius strip, widen the underlying  M\"obius strip in many instances.

For the remainder of Section 3, let $m \geq 7$.  For $\Sc{m}$ and $\Ma{m}{n}$, we define two disjoint subgraphs:
\begin{itemize}
    \item the inner graph consisting of all vertices $(a,b)$ with $1 < a < m-2$ and all edges between those vertices; and
    \item the outer graph consisting of all vertices that are not in the inner graph and all edges between those vertices.
\end{itemize}

In particular, the outer graph is the subgraph determined by the two outermost rows on each side of $\Ma{m}{n}$ or $\Sc{m}$, and we denote these $\ma{n}$ and $\Tma$ respectively; implicit in this notation is the observation that the structure of the outer graphs of $\Ma{m}{n}$ and $\Sc{m}$ are independent of $m$.  The graph $\Tma$ has 8 connected components, each of which is an infinite graph of vertices of valence 2.  Of these connected components, the 4 {\it left components} are in columns 0 and 1, and the 4 {\it right components} are in columns $m-2$ and $m-1$.  For each connected component, the $y$-coordinates of every vertex in the connected component have the same parity and so we name the components by this parity: there are two {\it even} left components and two {\it odd} left components. The same is true for right components. Frame (A) of Figure \ref{fig:wide} shows the inner and outer graphs of $\Sc{7}$ illustrated as the thick black and thick red subgraphs respectively.

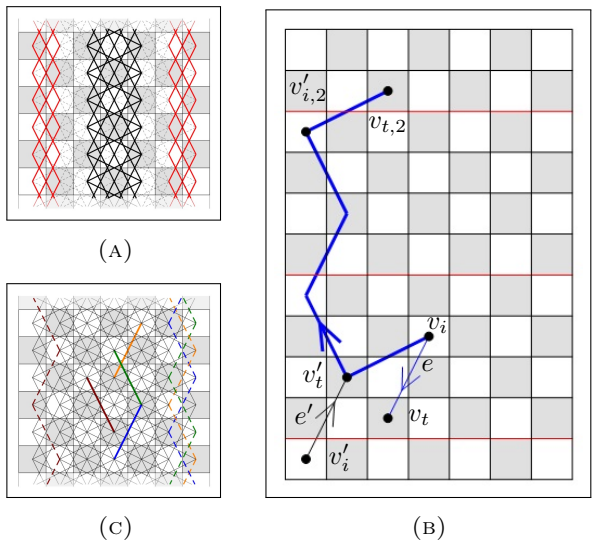
\begin{figure}[t]
\centering
\begin{picture}(300,206)
\put(38,105){\begin{subfigure}[b]{0.26\textwidth}
\fbox{\includegraphics[width = .9\textwidth]{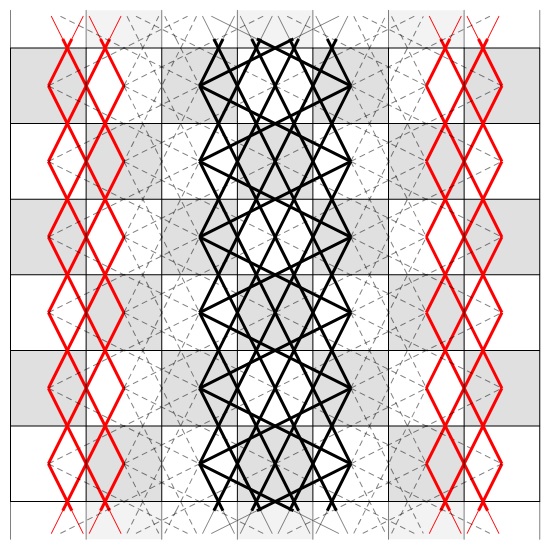}}
\caption{}
\end{subfigure}}
\put(136,0){\begin{subfigure}[b]{0.407\textwidth}
\fbox{\includegraphics[width = .9\textwidth]{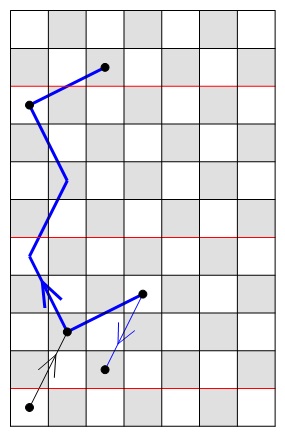}}
\caption{\:}
\end{subfigure}}
\put(38,0){\begin{subfigure}[b]{0.26\textwidth}
\fbox{\includegraphics[width = .9\textwidth]{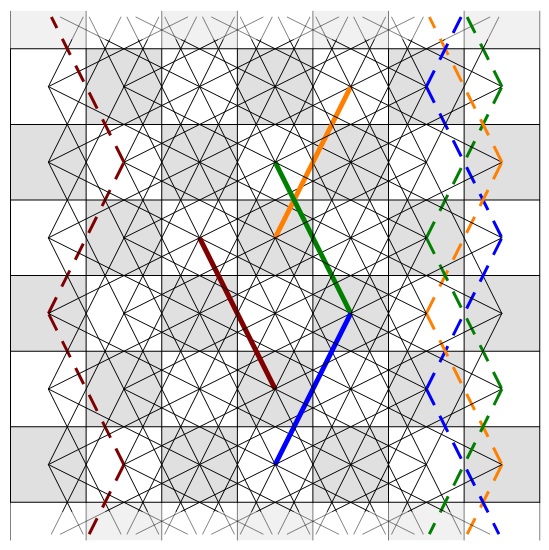}}
\caption{}
\end{subfigure}}

\put(160,30){$v'_i$}
\put(150,61){$v'_t$}
\put(197,81){$v_i$}
\put(190,46){$v_t$}
\put(145,170){$v'_{i,2}$}
\put(175,157){$v_{t,2}$}
\put(147,45){$e'$}
\put(195,65){$e$}

\end{picture}
\caption{Frame (A): A portion of $\Sc{7}$, where the inner graph consists of the thick black edges and the outer graph $\Tma$ consists of the thick red edges. Frame (B): A directed edge (the thin blue edge) in $\Sc{7}$, which we take as a preimage of an edge  of $\Ma{7}{4}$,  with corresponding replacement path (the thick blue edge path). The vertices and edges are labeled in agreement with the definition of replacement path.  Frame (C): An extending collection of edges in $\Sc{7}$ for even $n$.  The thick solid edges form an extending collection, and the images of the corresponding dashed edges under $\phi_M$ give a bijection between the edges in the extending collection and the connected components of $\ma{n}$.}
\label{fig:wide}
\end{figure}

The structure of $\ma{n}$ depends on the parity of $n$.  When $n$ is odd, $\ma{n}$ consists of two cycles of length $2n$ and the preimage of each cycle under $\phi_M$ consists of four components of $\Tma$; the left even and right odd components are the preimage of one cycle while the right even and left odd components are the preimage of the other cycle.  When $n$ is even, $\ma{n}$ is four disjoint cycles of length $n$, and the preimage of each cycle under $\phi_M$ consists of two components of $\Tma$; the preimages of two the components of $\ma{n}$ are made up of one left even and one right even component, while the preimages of the other two components are one left odd and one right odd component.  Let $N$ denote the length of the cycles on $\ma{n}$. 

Observe that $\Sc{m-4}$ and $\Ma{m-4}{n}$ include as the inner graphs of $\Sc{m}$ and $\Ma{m}{n}$ respectively and that these inclusions commute with the covering maps $\phi_M$. These inclusions have analogs for the surfaces $M$ and $S$.  Most importantly, the $(m-4) \times n$ M\"{o}bius strip includes in the $m \times n$ strip by the map $i(x,y) = (x+2,y)$.  This inclusion together with the inclusion of $\Ma{m-4}{n}$ as the inner graph of $\Ma{m}{n}$ commute with the inclusion maps of $\Ma{m-4}{n}$ and $\Ma{m}{n}$ into their corresponding M\"{o}bius strips. Lastly, $i$ induces an isomorphism on the corresponding fundamental groups of the M\"{o}bius strips.

By inclusion, a path in $\Sc{m-4}$ creates a path in the inner graph of $\Sc{m}$.  Below, we describe a splicing procedure that, given a path on the inner graph of $\Sc{m}$, creates a new extended path on $\Sc{m}$.  In particular, consider an edge $e$ in $\Sc{m}$ that is adjacent to vertices $(2,a)$ and $(3,a')$.  Note that $a$ and $a'$ have the same parity, and we will label the edge with that parity. For instance, if $a$ and $a'$ are even, we will call $e$ an {\it even edge}. We will also call $e$ a {\it left edge}; this same discussion applies to edges adjacent to vertices in columns $m-3$ and $m-4$, and we will call these {\it right edges}.

The edge $e$ is part of exactly two 4-cycles containing an edge in $\Tma$. Let $e'$ be one of those edges. Note that $e'$ is in a left component of $\Tma$ and that this component and edge $e$ have opposite parities.  We now give $e$ a direction and let $v_i$ and $v_t$ be the initial and terminal vertices of $e$ respectively. We say that $e$ is traversed {\it upward} if the second coordinate of $v_t$ is greater than the second coordinate of $v_i$ and {\it downward} otherwise.  This direction determines a direction on $e'$ as part of the 4-cycle; the direction on $e'$ is opposite the direction on $e$. Let $v'_i$ and $v'_t$ be the initial and terminal vertices of $e'$.

Now suppose that the original path in $\Sc{m-4}$ maps to a path on $\Ma{m}{n}$ under $\phi_M$. Removing all preimages of $\phi_M(e')$ from the connected component of $\Tma$ containing $e'$ creates a collection of length $N-1$ simple edge paths.  One of these paths contains $v'_t$ as an end point and a preimage of $\phi_M(v'_i)$, which we denote as $v'_{i,2}$, as the other endpoint.  Note that $v'_{i,2}$ shares an edge with the preimage of $\phi_M(v_t)$ given by the coordinates $(a,b + 2N)$ when $e$ is traversed downward and $(a, b-2N)$ when $e$ is traversed upward.  We denote this vertex $v_{t,2}$. The {\it replacement path} is the length $N+1$ edge path that traverses the edge from $v_i$ to $v'_t$, then the unique edge path in $\Tma$ from $v'_t$ to $v'_{i,2}$, and lastly the edge from $v'_{i,2}$ to $v_{t,2}$.  Replacement paths have the following properties:
\begin{itemize}
    \item The image of the replacement path under $\phi_M$ is a simple path in $\ma{n}$ that contains every vertex in the connected component of $\ma{n}$ that contains $\phi_M(e')$.
    \item The edge $\phi_M(e)$ is incident to $\phi_M(v_i)$ and $\phi_M(v_{t,2})$. 
\end{itemize}

\begin{example}
\label{ex:replacement}
Frame (B) of Figure \ref{fig:wide} shows a directed edge $e$ (the thin blue edge) in the inner graph of $\Sc{7}$ which we consider as a preimage of an edge in $\Ma{7}{4}$.  A replacement path for $e$ is shown by the thick blue edge path, and the edge $e'$ is the thin black edge.  The vertices $v_i$, $v_t$, $v'_i$, $v'_t$, $v'_{i,2}$, and $v_{t,2}$ are labeled.  The horizontal red lines are the preimage of the line $y=0$ in the $7 \times 4$ M\"{o}bius strip under $p_M$.  While we will generally construct replacement paths for edges whose images under $\phi_M$ are part of a knight's tour, this is not necessary.  Indeed, there is no knight's tour on $\Ma{3}{4}$ by Theorem \ref{thm:watkinsms}; however, we can still define a replacement path for $e$.

The majority of the figures in this work are paths in $\Sc{m}$ or $\Pc$ that map to a knight's tour in $\Ma{m}{n}$ or $\Ka{m}{n}$; the graphs $\Sc{m}$ and $\Pc$ are shown as subsets $S$ and $P$ respectively.  Red horizontal lines are the preimage of line $y=0$ in $M$ and $K$ under the covering maps $p_M$ and $p_K$ while blue vertical lines are the preimages of $x=0$ in $K$ under $p_K$.  That is, the red lines and blue lines partition the universal cover into fundamental domains.  The white dot is the starting point of the path and is typically $\overline{0}$.
\end{example}

\subsection{Extendable Tours}
In this subsection, we consider edge paths in the inner graph of $\Sc{m}$ whose image under $\phi_M$ is a knight's tour in $\Ma{m-4}{n}$.  We examine sets of left and right edges in $\Sc{m}$ that enable the use of replacement paths to extend the original edge path and create a knight's tour in $\Ma{m}{n}$. Recall that the image under $\phi_M$ of each replacement path contains exactly the vertices of a connected component of $\ma{n}$.  A collection of edges for which replacement paths can be selected so that each connected component of $\ma{n}$ is the image of exactly one replacement path is {\it extending}.  When $n$ is even, a collection of left and right edges is extending if and only if it consists of two even edges and two odd edges.  When $n$ is odd, a collection is extending if and only if it consists of exactly one edge that is left even or right odd and exactly one edge that is left odd or right even. Frame (C) of Figure \ref{fig:wide} shows an extending collection of edges in $\Sc{7}$ in the case when $n$ is even.  Generally, when the term extending is used, it will be clear from context if the collection of edges we are considering is extending for even $n$ or for odd $n$.

Now we can state our primary technical result enabling induction on the width of M\"{o}bius strip boards:

\begin{proposition}
Suppose that $f \colon (I,0) \to (\Ma{m-4}{n}, \OO)$ is a knight's tour with lift $\tilde{f} \colon (I,0) \to (\Sc{m-4}, \OO)$, where $\Sc{m-4}$ is the inner graph of $\Sc{m}$.  If $\tilde{f}$ traverses an extending collection of edges, half of which are traversed upward and the other half downward, then there exists a knight's tour $f'$ on $\Ma{m}{n}$ so that $\tilde{f}$ and $\tilde{f'}$ have the same starting and ending vertices.
\label{prop:extending}
\end{proposition}

\begin{proof}
    Consider the case when $n$ is odd and fix an extending collection of edges $e_1$ and $e_2$ traversed by $\tilde{f}$ in order. Let $v_0, v_1, \ldots v_{mn-4n}$ be the vertices traversed by $\tilde{f}$ in order.  We let $k_i$ be the index of the initial vertex of $e_i$; that is, $\tilde{f}$ traverses $e_i$ from $v_{k_i}$ to $v_{k_i + 1}$. Additionally, for each $e_i$, we choose a replacement path $R_i$ so that each connected component of $\ma{n}$ is the image of exactly one replacement path. We will construct $\tilde{f'}$ by replacing each edge in the extending collection with a replacement path.  Note that each instance of replacement applies a vertical shift to the remainder of the path. For example, when $e_1$ is replaced by $R_1$, we apply a vertical shift of $\pm 4n$ to the portion of $\tilde{f}$ traversed after $e_1$, where the direction of the shift is opposite from the direction in which $\tilde{f}$ traversed $e_1$.

    More formally, let $U_i$ be the number of edges in the extending collection that are traversed upward before $v_i$ is reached and let $D_i$ be the analogous value for edges in the collection that are traversed downward.  We define $P_0$ to be the edge path $\tilde{f}$ from $v_0$ to $v_{k_1}$, $P_1$ to be the edge path $\tilde{f}$ from $v_{k_1 + 1}$ to $v_{k_2}$, and $P_2$ to be the edge path from $v_{k_2 +1}$ to $v_{mn-4n}$. For $i = 1$ or $2$, we set $P'_i$ to be the edge path obtained by shifting $P_i$ vertically by $2N(D_{k_i + 1} - U_{k_i + 1})$ and $R'_i$ to be the edge path obtained by shifting $R_i$ vertically by $2N(D_{k_i} - U_{k_i})$. Then $\tilde{f'}$ is the concatenation of $P_0$, $R_1$, $P'_1$, $R'_2$, and $P'_2$. Note that $D_{k_2} = D_{k_{1} + 1}$ and $U_{k_2} = U_{k_{1} + 1}$, so the end point of $P'_{1}$ is the initial point of $R'_2$. Additionally, the end point of $R'_2$ is the initial point of $P'_{2}$ by the definition of replacement path. Note that $\tilde{f'}$ and $\tilde{f}$ start at the same point and that the ending vertices of the paths $\tilde{f}$ and $\tilde{f'}$ have been vertically shifted by $4n(D_{k_2 + 1} - U_{k_2+1})$.  However, $\tilde{f}$ traverses the same number of edges in the extending collection upward as it does downward, so $D_{k_2+1} = U_{k_2 + 1}$.  Hence, $\tilde{f}$ and $\tilde{f'}$ end at the same vertex.

    To see that $\phi_M(\tilde{f'}) = f'$ is a knight's tour, note first that $\phi_M(P_i) = \phi_M(P'_i)$ and that $f'$ traverses each vertex in $\Ma{m}{n} \setminus \ma{n}$ exactly once.  Further, each vertex in $\ma{n}$ is traversed exactly once by a path in $\phi_M( R_1 \cup R'_2 )$, which shows that $f'$ is a knight's tour.  Lastly, $\phi_M(\tilde{f})$ is a circuit, as $\tilde{f}$ and $\tilde{f'}$ share the same starting and ending vertices.  

    The case when $n$ is even is identical except that the extending collection consists of 4 edges and the vertical shifts occur in multiples of $2n$.
\end{proof}

\begin{definition}
    We say that a path $f \colon (I,0) \to (\Ma{m-4}{n}, \OO)$ meeting the criteria of Proposition \ref{prop:extending} is {\it extendable}.
\end{definition}

\begin{remark}
\label{rem:figures2}
Within the figures in this work, some of the preimages of knight's tours shown in $\Sc{m}$ have extending edge collections.  When an extending edge collection in a knight's tour is used in our arguments, the figure shows the collection in blue. If no extending edge collection has been identified, that does not mean that no such collection exists.
\end{remark}

\begin{remark}
\label{rem:ind}
Given an extendable path $f$, when the path $\tilde{f'}$ constructed in Proposition \ref{prop:extending} is taken as a path in the inner graph of $\Sc{m+4}$, it also traverses an extending collection of edges. The second edges traversed in each $R'_i$ form an extending collection, and $\tilde{f'}$ traverses half of these edges upward and half downward. This observation shows that $\tilde{f'}$ is extendable and provides an induction argument for Corollary \ref{cor:widen}.
\end{remark}

\begin{example}
\label{ex:widen4x4}
Figure \ref{fig:widenexample4x4} gives an example of Proposition \ref{prop:extending}.  Specifically, frame (A) of Figure \ref{fig:widenexample4x4} is a lift $\tilde{f}$ of a nullhomotopic knight's tour in $\Ma{4}{4}$ to $\Sc{4}$.  The blue edges are an extending collection in $\Sc{8}$.  Note that $\tilde{f}$ traverses the first two blue edges upward and the last two blue edges downward, so the original knight's tour is extendable.  Frame (B) of Figure \ref{fig:widenexample4x4} shows $\tilde{f'}$, the result of applying Proposition \ref{prop:extending} to frame (A).  That is, frame (B) is the preimage in $\Sc{8}$ of a nullhomotopic tour in $\Ma{8}{4}$.  The blue edges in frame (B) are the extending collection described in Remark \ref{rem:ind}.  Specifically, these edges are an extending collection in the inner graph of $\Sc{12}$. They enable the knight's tour in frame (A) to act as the base case for an induction argument creating a nullhomotopic tour in $\Ma{4 + 4k}{4}$, where $k$ is any non-negative integer.  Note that the tour created by Proposition \ref{prop:extending} depends on the base point, and we denote this point by the white dot; in frame (A), this is $(0,0)$, while in frame (B), this is $(2,0)$.  As we discuss in Corollary \ref{cor:widen} and Remark \ref{rem:basepoint}, we can alter the starting point of $f'$ without changing the topology of the tour.   
\end{example}

\begin{figure}[t]
\centering
\begin{subfigure}[b]{0.2\textwidth}
\fbox{\includegraphics[width = .9\textwidth]{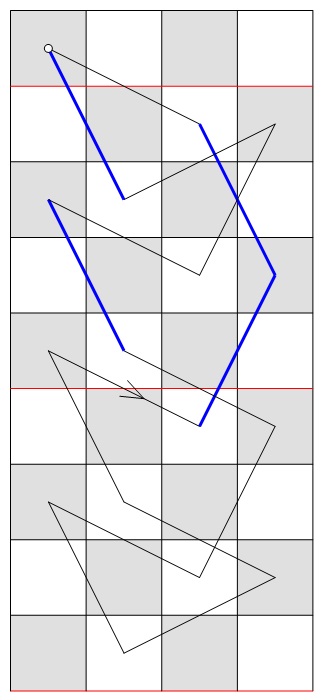}}
\caption{}
\end{subfigure}
\quad \qquad
\begin{subfigure}[b]{0.22\textwidth}
\fbox{\includegraphics[width = .9\textwidth]{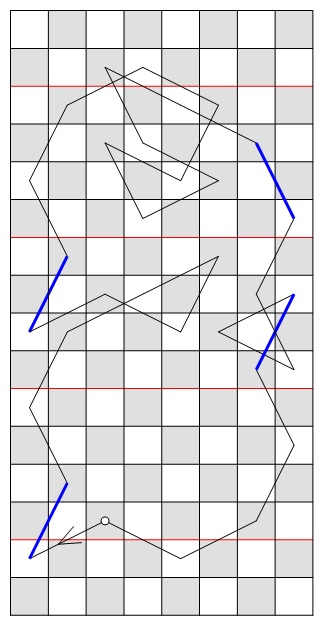}}
\caption{}
\end{subfigure}
\caption{Lifts of nullhomotopic tours in $\Ma{4}{4}$ and $\Ma{8}{4}$ discussed in Example \ref{ex:widen4x4}.  Frame (A): The preimage in $\Sc{4}$ of a nullhomotopic tour in $\Ma{4}{4}$ where the white dot is $(0,0)$ and the blue egdes are an extending collection in $\Sc{8}$.  Frame (B): The preimage in $\Sc{8}$ of a nullhomotopic tour in $\Ma{8}{4}$ with base point given by white dot at $(2,0)$.  This is the result of applying the methods of  Proposition \ref{prop:extending} to the path in Frame (A).  The blue edges are the extending collection in $\Sc{12}$ described in Remark \ref{rem:ind}.}
\label{fig:widenexample4x4}
\end{figure}

\begin{corollary}
\label{cor:widen}
Let $f \colon (I,0) \to (\Ma{m-4}{n},\OO)$ be an extendable knight's tour.  
\begin{itemize}
    \item If $f$ is nullhomotopic, then there exists a nullhomotopic tour on $\Ma{m+4k}{n}$ for each non-negative integer $k$.
    \item If $f$ is generating, then there exists a generating tour on $\Ma{m+4k}{n}$ for each non-negative integer $k$.
\end{itemize}
\end{corollary}

\begin{proof}
We identify $\Sc{m+4k}$ with the inner graph of $\Sc{m+4k+4}$.  Note that by Proposition \ref{prop:extending} and Remark \ref{rem:ind}, for each $k$ there exists a knight's tour $f_k$ on $\Ma{m+4k}{n}$ so that $\tilde{f}_k$ has the same starting and ending points as $\tilde{f}$.  By Theorem \ref{thm:cover}, $p_M \circ i_S(\tilde{f})$ and $p_M \circ i_S(\tilde{f}_k)$ are path homotopic in $M$, so $i_M(f)$ and $i_M(f_k)$ are path homotopic in $M$ where $i_M$ is the inclusion of $\Ma{m+4k}{n}$ into the $(m+4k)\times n$ M\"{o}bius strip $M$.  The inclusion of the $(m-4) \times n$ M\"{o}bius strip into $M$ induces an isomorphism on the fundamental groups, as it is a composition of $k + 1$ isomorphisms. We see that $[i_M(f_k)]$ is equal to $[i_M(f)]$ in $\pi_1(M,i_M(2(k+1),0))$.  Since $M$ is path-connected, there is an isomorphism from $\pi_1(M,i_M(2(k+1),0))$ to $\pi_1(M,v)$ given by concatenating loops based at $i_M(2(k+1),0)$ with paths and inverse paths from $i_M(2(k+1),0)$ to $v = i_M(0,0)$.  The image of the initial subpath of $f_k$ from $(2(k+1),0)$ to $(0,0)$ under $i_M$ is such a path, and the result of this concatenation is homotopic to $i_M(f_k)$ with initial point $v$.  Hence, taking the initial point of $f_k$ to be $(0,0)$ in $\Ma{m+4k}{n}$ produces a tour that is nullhomotopic if $f$ is nullhomotopic and generating if $f$ is generating.
\end{proof}

\begin{remark}
\label{rem:basepoint}
In Corollary \ref{cor:widen}, consider the case when $f$ is generating.  In this case, note that $\tilde{f}_k$ begins at $(2(k+1),0)$ and ends at $(m+2k-2,n)$.  Let $P_1$ be the subpath starting at $(2(k+1),0)$ and ending at a preimage $(a,b)$ of $(0,0)$ under $p_M$, and let $P_2$ be the remainder of $f_k$.  We can construct a generating tour based at $(0,0)$ by concatenating $P'_2$ followed by $P'_1$.  Here, $P'_2$ is given by transforming $P_2$ to start at $(0,0)$; this transformation is a translation by $(0,-b)$ when $a = 0$ and a flip about the vertical line $y = (m + 4k - 1)/2$ followed by the translation by $(0,-b)$ when $a = m + 4k -1$.  Similarly, $P'_1$ is given by transforming $P_1$ to start at the end point of $P'2$.  When $a = m + 4k -1$, we translate $P_1$ by $(0,-b+n)$, while when $a=0$, we apply a reflection about the vertical line $y = (m + 4k - 1)/2$ followed by translation by $(0,-b+n)$. This new edge path has the same image under $\phi_M$ as $\tilde{f}_k$, starts at $(0,0)$, and ends at $(m+4k-1,n)$. Hence, this is a generating tour.
\end{remark}

\section{\bf Nullhomotopic Tours on M\"{o}bius Strips}

In this section, we characterize the dimensions of M\"{o}bius strips that admit nullhomotopic knight's tours.

\begin{theorem}
\label{thm:MSNH}
The multigraph $\Ma{m}{n}$ supports a nullhomotopic tour if and only if none of the following are true:
\begin{itemize}
\item $m$ and $n$ are both odd and at least one of $m$ and $n$ is greater than 1;
\item $m = 1$ and $n > 1$;
\item $m = 2$;
\item $m = 3$ and $n = 4$; or
\item $m = 4$ and $n$ is odd.
\end{itemize}
\end{theorem}

When both $m$ and $n$ are greater than 4, Theorem \ref{thm:schwenk} implies that if at least one of $m$ or $n$ is even, then a nullhomotopic tour exists on $\Ma{m}{n}$, since there is a knight's tour on $\Ra{m}{n}$.  Further, Proposition \ref{prop:oddeven} implies that if both $m$ and $n$ are odd, then there is no nullhomotopic tour on $\Ma{m}{n}$. This establishes Theorem \ref{thm:MSNH} in the case that $m$ and $n$ are both greater than 4.  Indeed, in this case, $\Ma{m}{n}$ supports a nullhomotopic tour if and only if there is a knight's tour on $\Ra{m}{n}$.  This is not true for smaller boards; the edges that are added to $\Ra{m}{n}$ to create $\Ma{m}{n}$ enable nullhomotopic tours that are impossible otherwise. To prove Theorem \ref{thm:MSNH}, we consider 8 cases by fixing $m$ to be 1, 2, 3, or 4 and likewise fixing $n$.

\subsection*{{\bf $1 \times n$}}
There is a nullhomotopic tour on $\Ma{1}{1}$, as it has only one vertex. Otherwise, there is no knight's tour on $\Ma{1}{n}$ by Theorem \ref{thm:watkinsms}.

\subsection*{{\bf $m \times 1$}}
The multigraph $\Ma{2}{1}$ has two vertices connected by two edges, and neither of the two possible knight's tours are nullhomotopic.  For odd values of $m$ with $m \geq 3$, by Proposition \ref{prop:oddeven}, there is no nullhomotopic tour on $\Ma{m}{1}$.  By Theorem \ref{thm:watkinsms}, there are no knight's tours on $\Ma{4}{1}$.  

This leaves $\Ma{m}{1}$ where $m$ is even and $m \geq 6$.  Observe that Figure \ref{fig:MSNH1xn} gives preimages under $\phi_M$ of nullhomotopic tours on $\Ma{6}{1}$, $\Ma{8}{1}$, and $\Ma{10}{1}$.  We will show how to inductively use the tour on $\Ma{m}{1}$ to build a nullhomotopic tour on $\Ma{m+2}{1}$.  Formally, we map the preimage of the tour under $\phi_M$ to $\Sc{m+2}$ by the inclusion map from $\Sc{m}$ to $ \Sc{m+2}$ sending vertex $(a,b)$ to $(a,b)$ and the edge from $(a,b)$ to $(a+c,b+d)$ to itself where $\{|c|,|d|\} = \{1,2\}$.  In practice, we can visually consider, for example, the path shown in frame (A) of Figure \ref{fig:MSNH1xn} in $\Sc{8}$ and likewise consider the path shown in frame (B) as a path in $\Sc{10}$.  We will delete one edge from each of these paths and replace it with three edges; those four edges - the one deleted edge and the three new edges - form a 4-cycle.  In particular, when $m$ is equivalent to 2 modulo 4, we delete the edge from $\left( \frac{m}{2}-2, -2 \right)$ to $\left( \frac{m}{2}-1, 0 \right)$ and replace it with the three-edge path $\left( \frac{m}{2}-2, -2 \right) - \left( \frac{m}{2}, -3 \right) - \left( \frac{m}{2} + 1,  -1 \right) - \left( \frac{m}{2}-1, 0 \right)$.  Applying this replacement to frame (A) of Figure \ref{fig:MSNH1xn} produces frame (B).  Updating our index by adding 2 to $m$, note now that frame (B) does not include an edge $\left( \frac{m}{2}-2, -2 \right)$ to $\left( \frac{m}{2}-1, 0 \right)$, so we cannot apply the procedure described above.  Instead, when $m$ is equivalent to 0 modulo 4, we delete the edge from $\left( \frac{m}{2}-3, -3 \right)$ to $\left( \frac{m}{2}-2, -1 \right)$ and replace it with the three-edge path $\left( \frac{m}{2}-3, -3 \right) - \left( \frac{m}{2}-1, -2 \right) - \left( \frac{m}{2},  0 \right) - \left( \frac{m}{2}-2, -1 \right)$.  Applying this replacement to frame (B) of Figure \ref{fig:MSNH1xn} produces frame (C).  This induction procedure produces a nullhomotopic tour on $\Ma{m}{1}$ when $m$ is even and $m \geq 6$.

\begin{figure}[t]

\centering
\begin{subfigure}[b]{0.3\textwidth}
\fbox{\includegraphics[width = .9\textwidth]{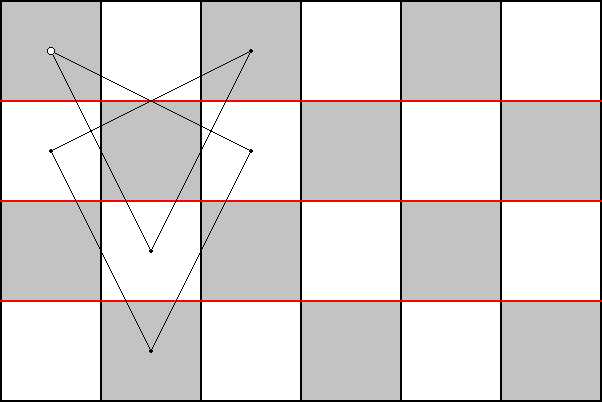}}
\caption{}
\end{subfigure}
\qquad
\begin{subfigure}[b]{0.4\textwidth}
\fbox{\includegraphics[width = .9\textwidth]{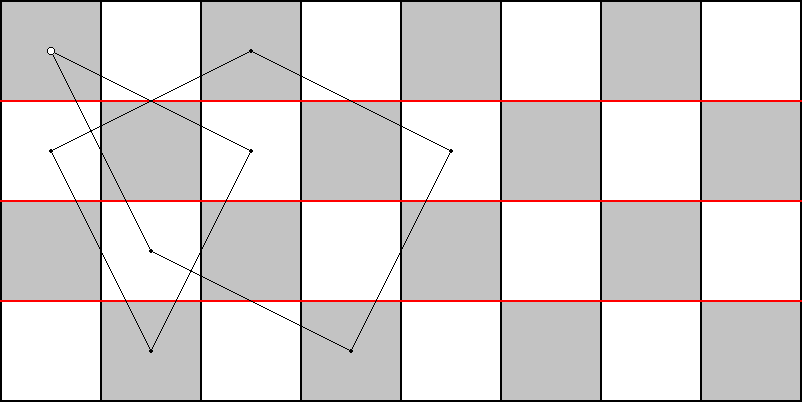}}
\caption{\:}
\end{subfigure}\\[8 pt]
\begin{subfigure}[b]{0.7\textwidth}
\fbox{\includegraphics[width = .9\textwidth]{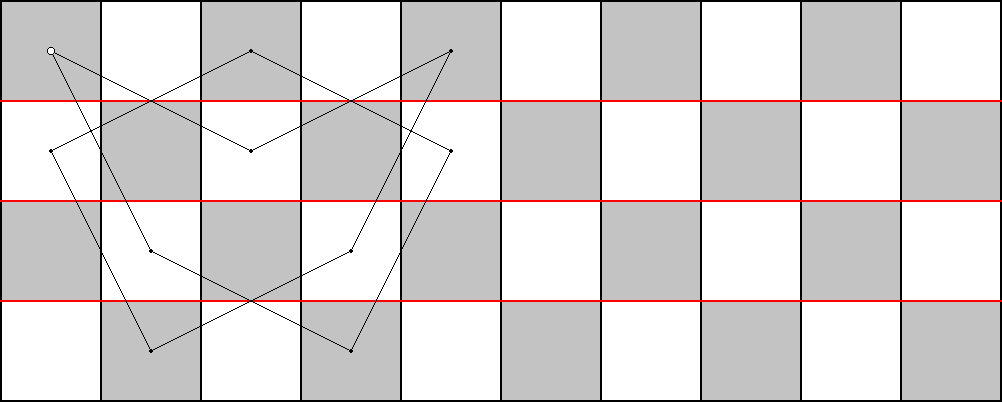}}
\caption{\;\;\;\:}
\end{subfigure}
\caption{Lifts of nullhomotopic tours in $\Ma{6}{1}$, $\Ma{8}{1}$, and $\Ma{10}{1}$.}
\label{fig:MSNH1xn}
\end{figure}

For the remainder of this section, we assume $m$ and $n$ are greater than $1$.

\subsection*{{\bf $2 \times n$}}
When $n$ is even, there is no knight's tour on $\Ma{2}{n}$ by Theorem \ref{thm:watkinsms}.  When $n$ is odd, $\Ma{2}{n}$ is a $2n$-cycle and neither of the two knight's tours are nullhomotopic.

\subsection*{{\bf $m \times 2$}}
Figure \ref{fig:MSNHmx2} shows the preimages under $\phi_M$ of nullhomotopic tours on $\Ma{m}{2}$ for each $m$ with $3 \leq m \leq 6$ and $m = 8$.  Note that for $m = 3, 5, 6,$ and $8$, these tours are extendable.  By Corollary \ref{cor:widen}, for each $m \geq 3$, there is a nullhomotopic tour on $\Ma{m}{2}$.

For the remainder of this section, we assume $m$ and  $n$ are greater than $2$.

\begin{figure}[ht]

\begin{subfigure}[b]{0.21\textwidth}
\fbox{\includegraphics[width = .9\textwidth]{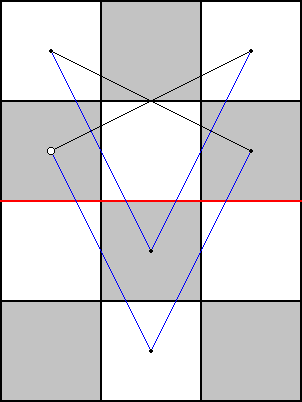}}
\caption{}
\end{subfigure}
\qquad
\begin{subfigure}[b]{0.225\textwidth}
\fbox{\includegraphics[width = .9\textwidth]{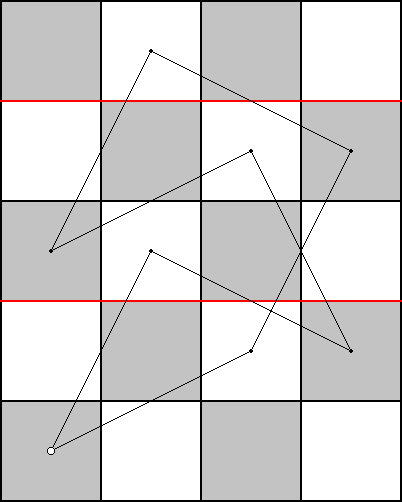}}
\caption{}
\end{subfigure}
\qquad
\begin{subfigure}[b]{0.35\textwidth}
\fbox{\includegraphics[width = .9\textwidth]{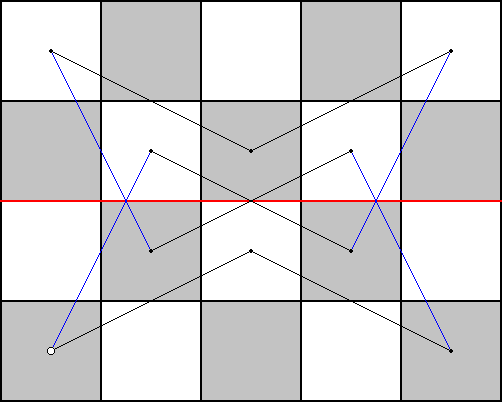}}
\caption{\:}
\end{subfigure}\\[6 pt]

\begin{subfigure}[b]{0.38\textwidth}
\fbox{\includegraphics[width = .9\textwidth]{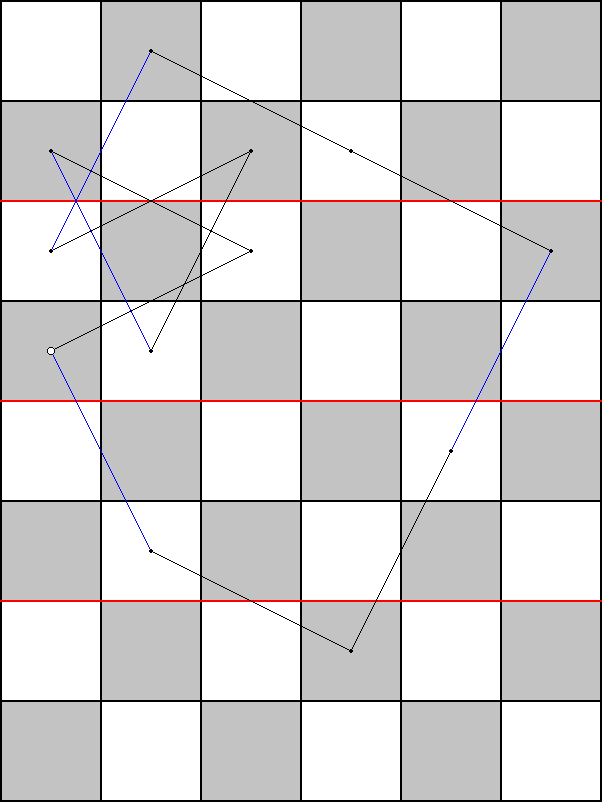}}
\caption{\:}
\end{subfigure}
\qquad
\begin{subfigure}[b]{0.5\textwidth}
\fbox{\includegraphics[width = .9\textwidth]{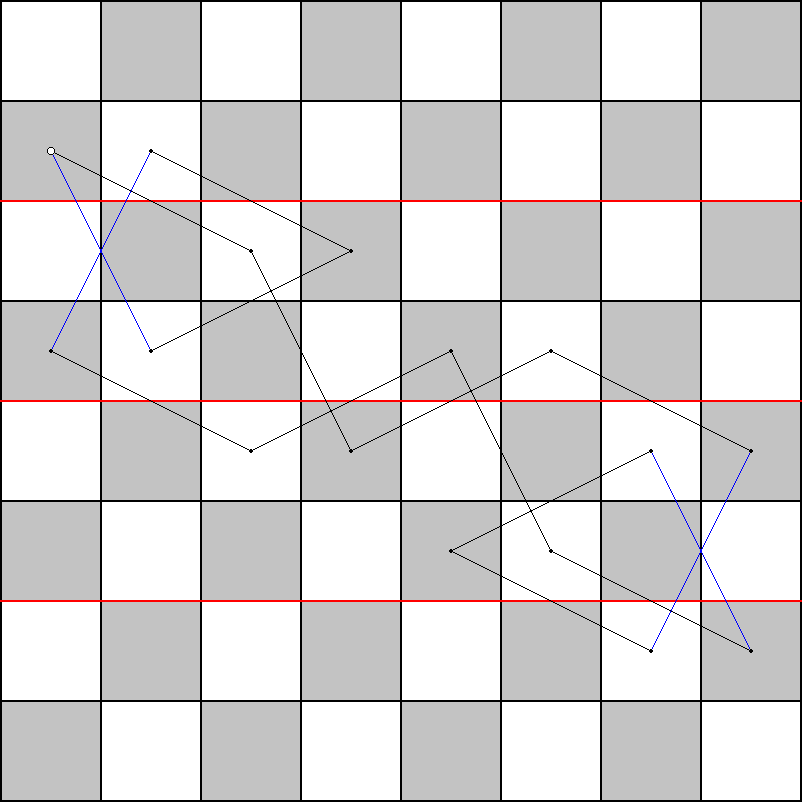}}
\caption{\;\:}
\end{subfigure}
\caption{Lifts of nullhomotopic tours in $\Ma{3}{2}$, $\Ma{4}{2}$, $\Ma{5}{2}$, $\Ma{6}{2}$, and $\Ma{8}{2}$.}
\label{fig:MSNHmx2}
\end{figure}

\subsection*{{\bf $3 \times n$}}
When $n$ is odd, there is no nullhomotopic tour on $\Ma{3}{n}$ by Proposition \ref{prop:oddeven}.  By Theorem \ref{thm:watkinsms}, there is no knight's tour on $\Ma{3}{4}$. By Theorem \ref{thm:schwenk}, the graph $\Ra{3}{n}$ has a tour when $n$ is even and greater than 9, so $\Ma{3}{n}$ has a nullhomotopic tour in those cases.  It remains to consider $\Ma{3}{6}$ and $\Ma{3}{8}$. Figure \ref{fig:MSNH3xn} shows the preimages under $\phi_M$ of nullhomotopic tours on $\Ma{3}{6}$ and $\Ma{3}{8}$.

\begin{figure}[t]

\centering
\begin{subfigure}[b]{0.17\textwidth}
\fbox{\includegraphics[width = .85\textwidth]{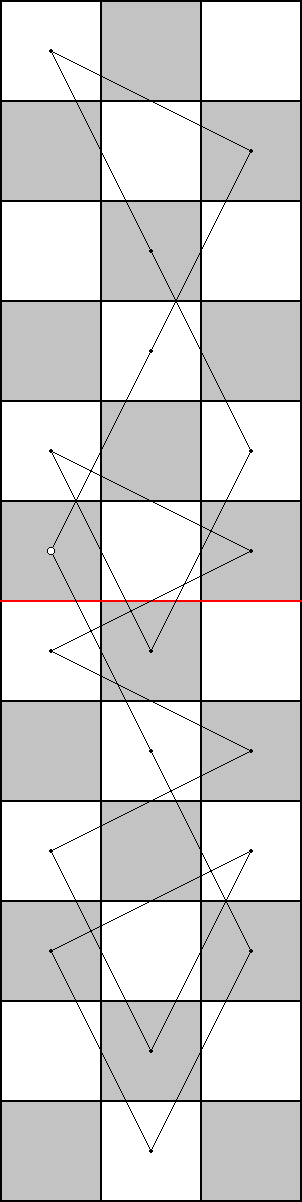}}
\caption{}
\end{subfigure}
\qquad \qquad
\begin{subfigure}[b]{0.17\textwidth}
\fbox{\includegraphics[width = .85\textwidth]{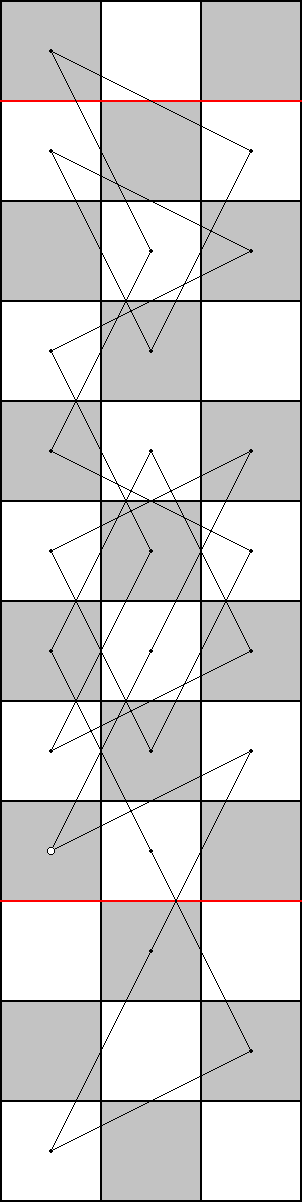}}
\caption{}
\end{subfigure}
\caption{Lifts of nullhomotopic tours in $\Ma{3}{6}$ and $\Ma{3}{8}$.}
\label{fig:MSNH3xn}
\end{figure}

\subsection*{{\bf $m \times 3$}}
By Theorem \ref{thm:watkinsms}, there are no knight's tours on $\Ma{4}{3}$.  By Proposition \ref{prop:oddeven}, there are no nullhomotopic tours on $\Ma{m}{3}$ when $m$ is odd.  Figure \ref{fig:MSNHmx3} shows the preimages under $\phi_M$ of extendable nullhomotopic tours on $\Ma{6}{3}$ and $\Ma{8}{3}$. By Corollary \ref{cor:widen}, there is a nullhomotopic tour on $\Ma{m}{3}$ when $m \geq 6$ and $m$ is even. Additionally, when $m$ is even and $m \geq 10$, there is a tour on $\Ra{m}{3}$ by Theorem \ref{thm:schwenk}, which is a nullhomotopic tour in $\Ma{m}{3}$.

\begin{figure}[t]

\centering
\begin{subfigure}[b]{0.4\textwidth}
\fbox{\includegraphics[width = .9\textwidth]{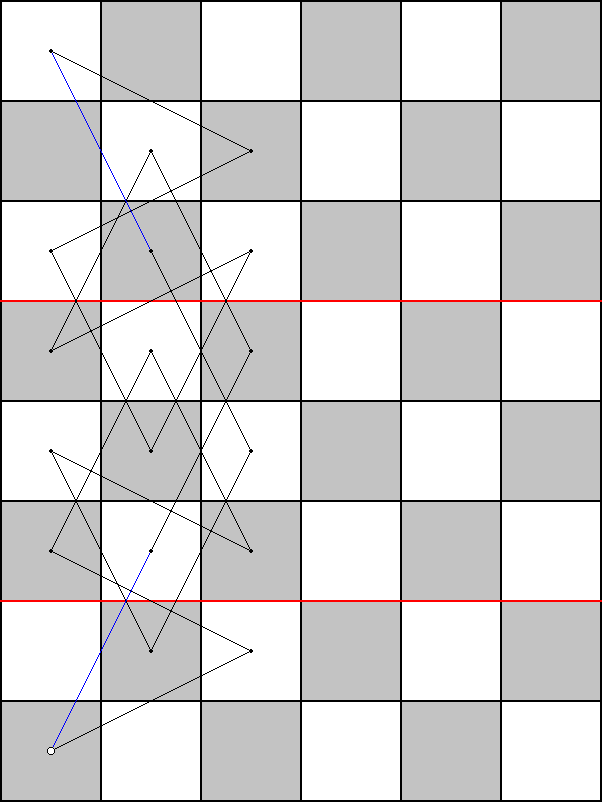}}
\caption{\;}
\end{subfigure}
\qquad \qquad
\begin{subfigure}[b]{0.35\textwidth}
\fbox{\includegraphics[width = .9\textwidth]{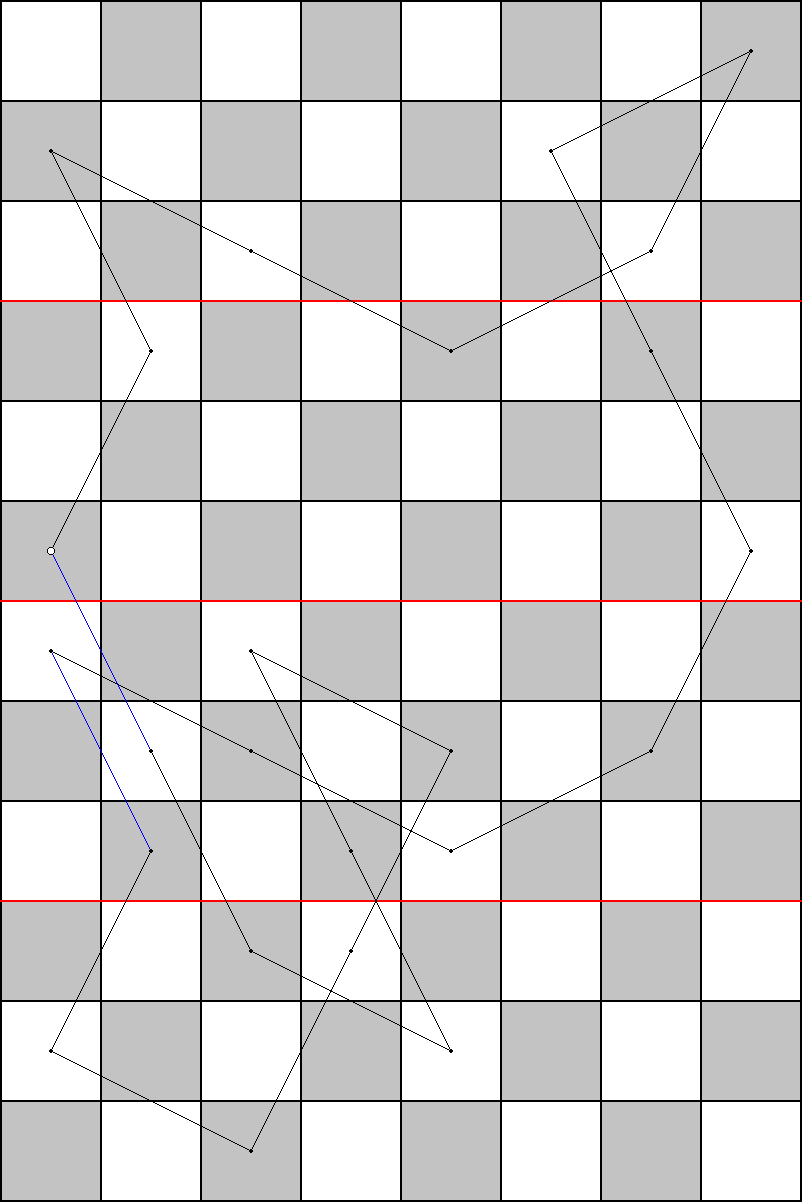}}
\caption{\,}
\end{subfigure}
\caption{Lifts of nullhomotopic tours in $\Ma{6}{3}$ and $\Ma{8}{3}$.}
\label{fig:MSNHmx3}
\end{figure}

For the remainder of this section, we assume $m$ and  $n$ are greater than $3$.

\subsection*{{\bf $4 \times n$}}
By Theorem \ref{thm:watkinsms}, there is no knight's tour on $\Ma{4}{n}$ when $n$ is odd.  Frame (A) of Figure \ref{fig:widenexample4x4} is the preimage under $\phi_M$ of a nullhomotopic tour on $\Ma{4}{4}$. This tour can be inductively extended downward to create tours on $\Ma{4}{n}$ for each even $n$ with $n \geq 6$.  In particular, suppose the lowest edge in the preimage of the tour on $\Ma{4}{n}$ in $\Sc{4}$ is from $(1,-2n)$ to $(3,-2n+1)$.  We create a path on $\Sc{4}$ by replacing this edge with the path of 9 edges given by traversing the following list of vertices: 

\begin{center}
    $(1,-2n),(3,-2n-1),(2,-2n-3),(0,-2n-2),(1,-2n-4),(3,-2n-3),$\\
$(1,-2n-2),(0,-2n),(2,-2n-1),(3,-2n+1)$
\end{center}
The image of this new path in $\Ma{4}{n+2}$ is a nullhomotopic tour. The result of applying this method to the preimage under $\phi_M$ of the nullhomotopic tour on $\Ma{4}{4}$ in frame (A) of Figure \ref{fig:widenexample4x4} is shown in frame (A) of Figure \ref{fig:MSNH4xn}.  Applying this argument inductively to frame (A) of Figure \ref{fig:MSNH4xn} gives the path in frame (B) of Figure \ref{fig:MSNH4xn}. 

\begin{figure}[t]
\centering
\begin{subfigure}[b]{0.22\textwidth}
\fbox{\includegraphics[width = .85\textwidth]{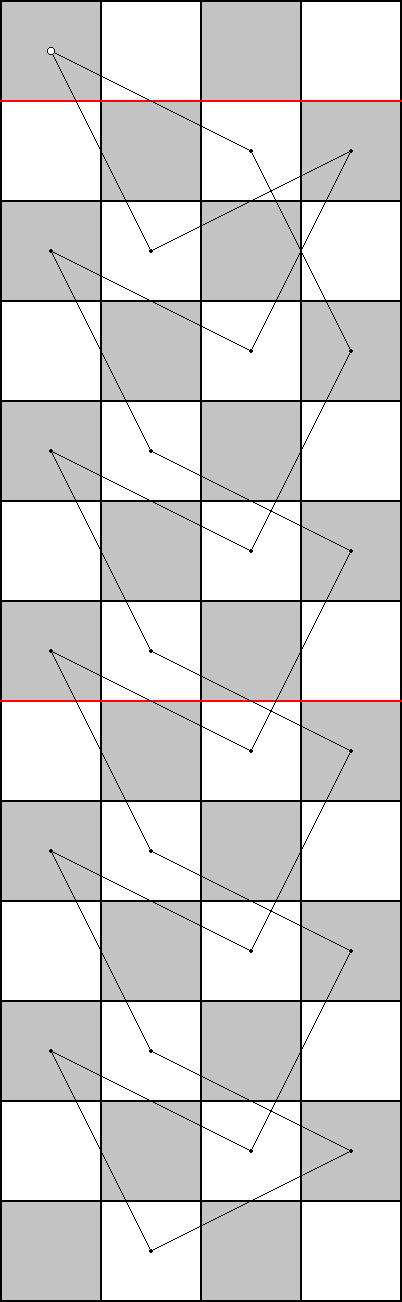}}
\caption{\,}
\end{subfigure}
\qquad \qquad
\begin{subfigure}[b]{0.17\textwidth}
\fbox{\includegraphics[width = .85\textwidth]{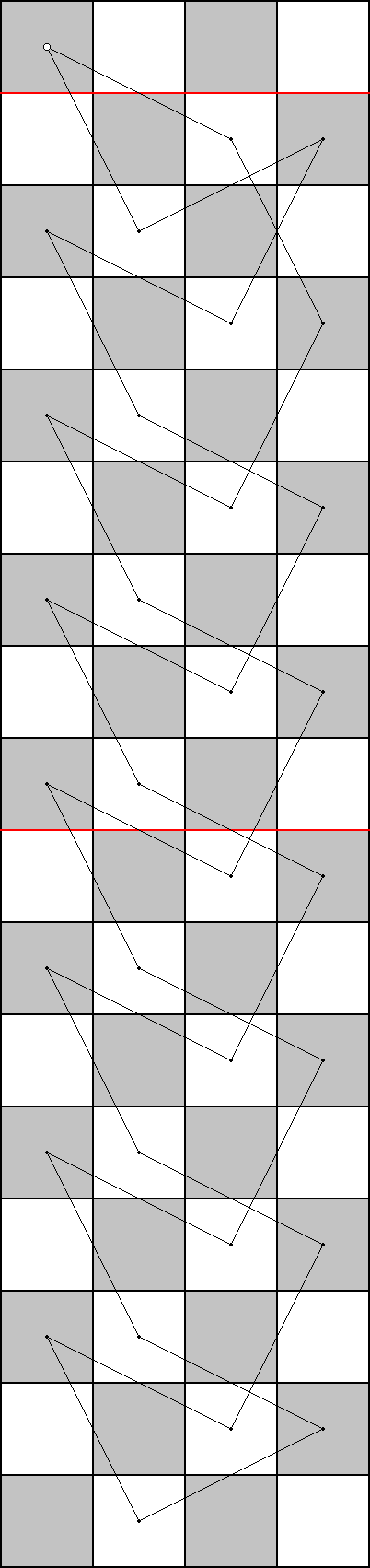}}
\caption{}
\end{subfigure}
\caption{Lifts of nullhomotopic tours on $\Ma{4}{6}$ and $\Ma{4}{8}$, created by applying the induction argument described in the $4 \times n$ case above to lift of the tour on $\Ma{4}{4}$ in frame (A) Figure \ref{fig:widenexample4x4}.}
\label{fig:MSNH4xn}
\end{figure}

\subsection*{{\bf $m \times 4$}}
Frame (A) of Figure \ref{fig:widenexample4x4} along with Figure \ref{fig:MSNHmx4} show preimages under $\phi_M$ of extendable nullhomotopic tours on $\Ma{4}{4}$, $\Ma{5}{4}$, $\Ma{6}{4}$, and $\Ma{7}{4}$.  By Corollary \ref{cor:widen}, there is a nullhomotopic tour on $\Ma{m}{4}$ for each $m \geq 4$.

This completes the proof of Theorem \ref{thm:MSNH}.

\begin{figure}[t]

\centering
\begin{subfigure}[b]{0.25\textwidth}
\fbox{\includegraphics[width = .85\textwidth]{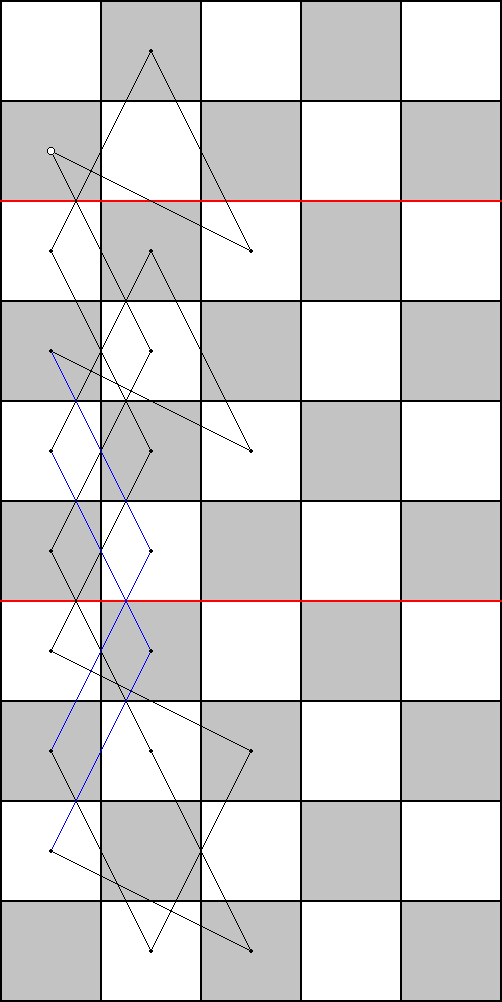}}
\caption{\,}
\end{subfigure}
\qquad
\begin{subfigure}[b]{0.25\textwidth}
\fbox{\includegraphics[width = .85\textwidth]{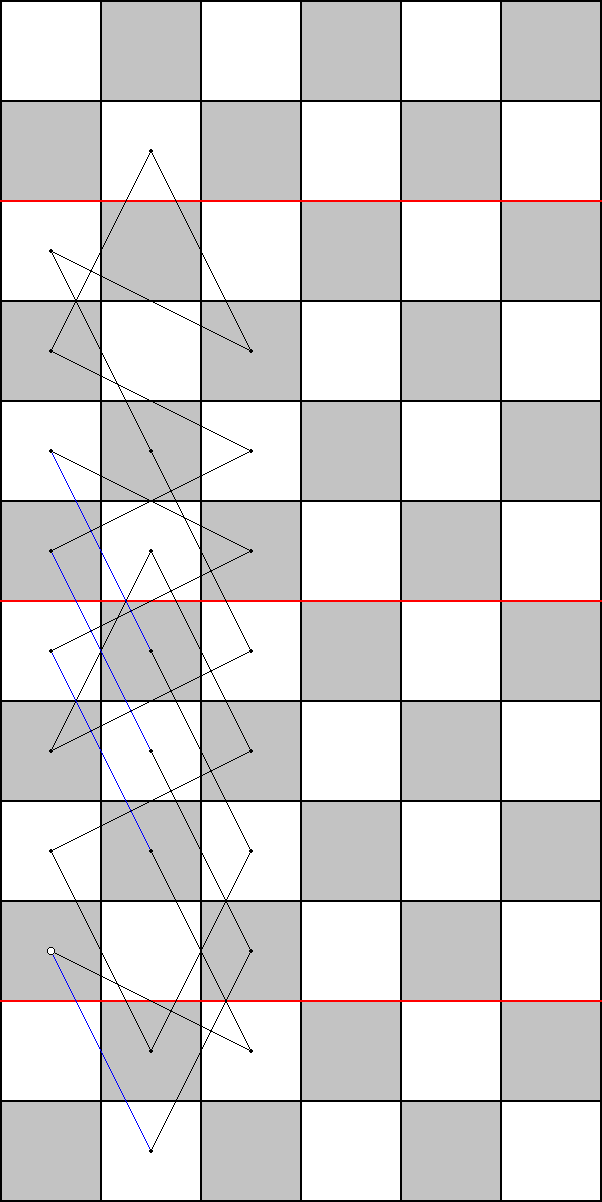}}
\caption{\,}
\end{subfigure}
\qquad
\begin{subfigure}[b]{0.25\textwidth}
\fbox{\includegraphics[width = .85\textwidth]{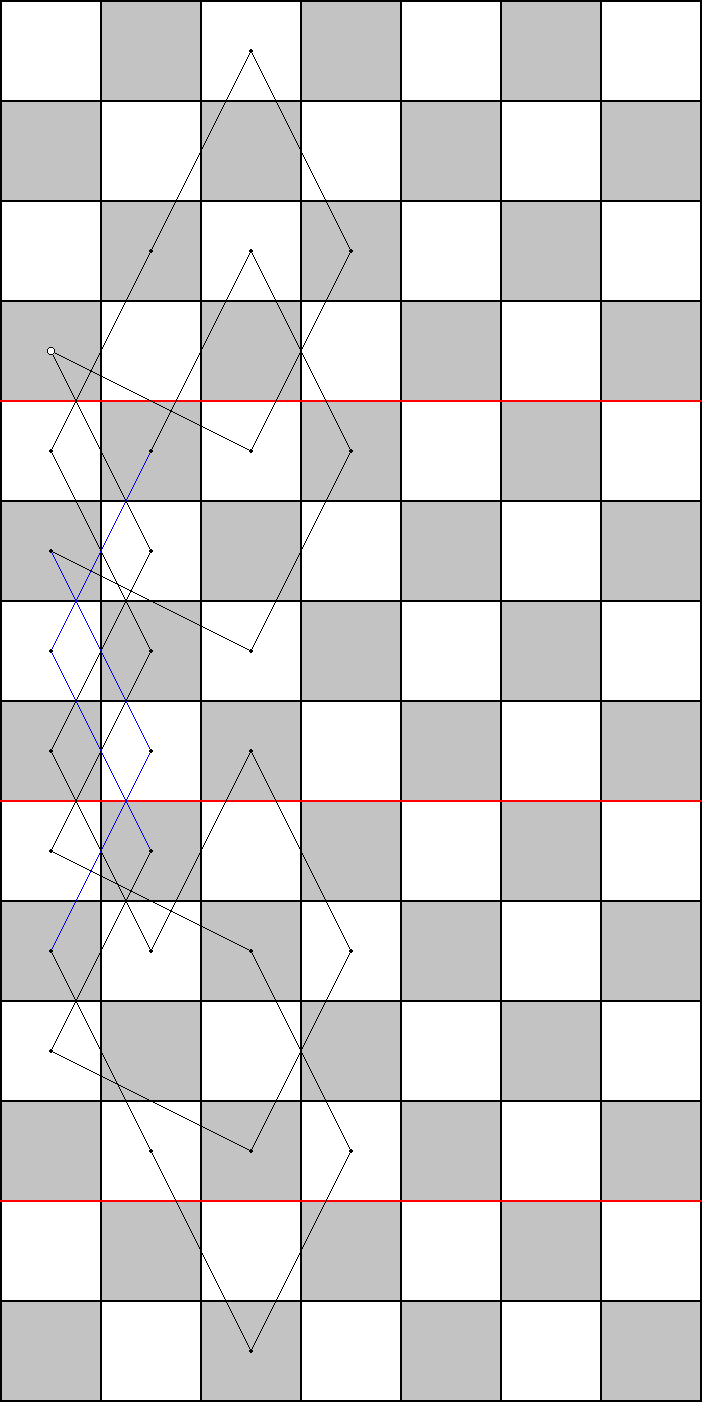}}
\caption{\,}
\end{subfigure}
\caption{Lifts of nullhomotopic tours in $\Ma{5}{4}$, $\Ma{6}{4}$, and $\Ma{7}{4}$.}
\label{fig:MSNHmx4}
\end{figure}

\section{\bf Generating Tours on M\"{o}bius Strips}

We now characterize the dimensions of M\"{o}bius strips that admit generating knight's tours.

\begin{theorem}
\label{thm:MSGen}
The multigraph $\Ma{m}{n}$ supports a generating tour if and only if none of the following are true:
\begin{itemize}
\item $m$ and $n$ are both even;
\item $m = 1, 2,$ or $4$;
\item $m = 3$ and $n = 1, 2,$ or $4$; or
\item $m = 5$ and $n = 1$.
\end{itemize}
\end{theorem}

We will prove Theorem \ref{thm:MSGen} by considering cases following the structure of the previous section, incrementing $m$ and $n$ as we proceed.

\subsection*{{\bf $1 \times n$}} Since the connected components of $\Sc{1}$ are single vertices, there is no generating tour in $\Ma{1}{1}$. When $n > 2$, there is no generating tour on $\Ma{1}{n}$ by Theorem \ref{thm:watkinsms}.  

\subsection*{{\bf $m \times 1$}} The multigraph $\Ma{2}{1}$ has two vertices connected by two edges, and neither of the two possible knight's tours are generating.  By Theorem \ref{thm:watkinsms}, there are no generating tours on $\Ma{3}{1}$, $\Ma{4}{1}$, or $\Ma{5}{1}$.  Figure \ref{fig:MSGenmx1} shows extendable generating tours on $\Ma{6}{1}$, $\Ma{7}{1}$, $\Ma{8}{1}$, and $\Ma{9}{1}$.  By Corollary \ref{cor:widen}, for each $m \geq 6$, there is a generating tour on $\Ma{m}{1}$.

For the remainder of this section, we assume $m$ and $n$ are greater than $1$.

\begin{figure}[t]

\centering
\begin{subfigure}[b]{0.3\textwidth}
\fbox{\includegraphics[width = .85\textwidth]{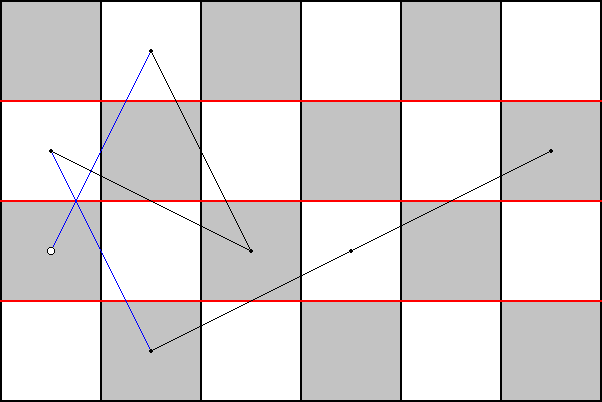}}
\caption{\;\:}
\end{subfigure}
\qquad
\begin{subfigure}[b]{0.34\textwidth}
\fbox{\includegraphics[width = .85\textwidth]{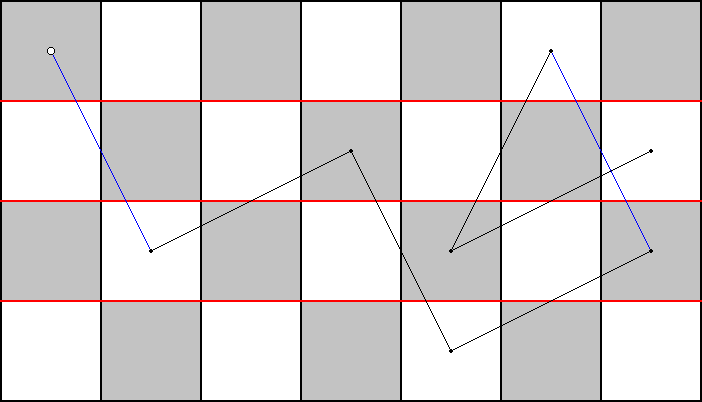}}
\caption{\:\;}
\end{subfigure}\\[8 pt]
\begin{subfigure}[b]{0.32\textwidth}
\fbox{\includegraphics[width = .85\textwidth]{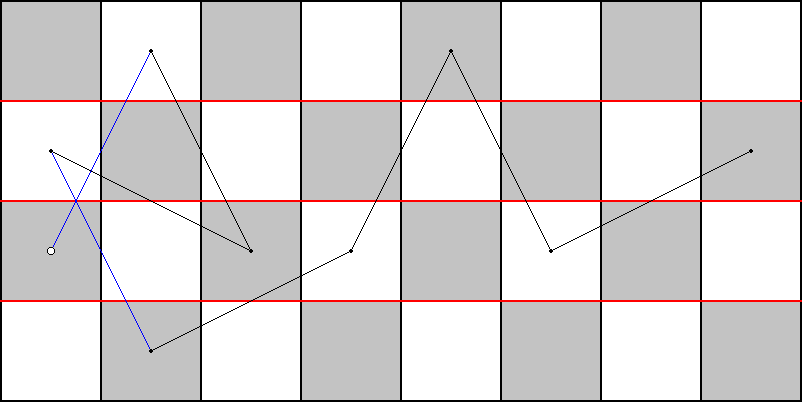}}
\caption{\,}
\end{subfigure}
\qquad
\begin{subfigure}[b]{0.35\textwidth}
\fbox{\includegraphics[width = .85\textwidth]{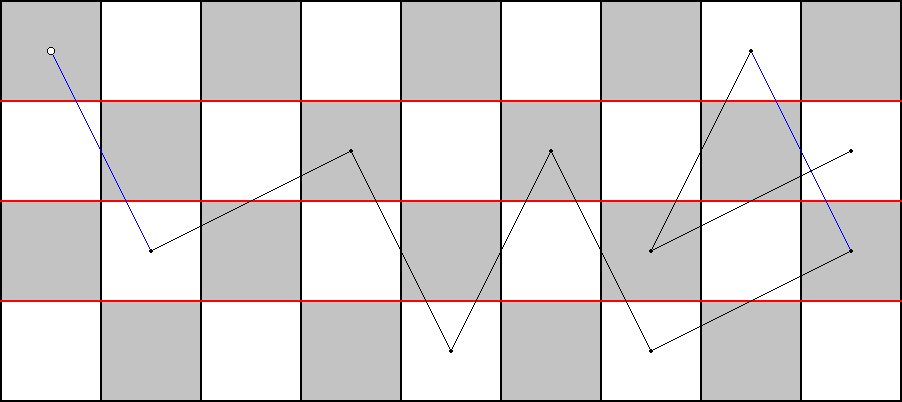}}
\caption{\;\;}
\end{subfigure}
\caption{Lifts of generating tours on $\Ma{6}{1}$, $\Ma{7}{1}$, $\Ma{8}{1}$, and $\Ma{9}{1}$.}
\label{fig:MSGenmx1}
\end{figure}

\subsection*{{\bf $2 \times n$}}
When $n$ is even, there is no knight's tour on $\Ma{2}{n}$ by Theorem \ref{thm:watkinsms}.  When $n$ is odd, $\Ma{2}{n}$ is a $2n$-cycle and neither of the two knight's tours are generating.

\subsection*{{\bf $m \times 2$}} We consider lifts to $\Sc{3}$ of knight's tours in $\Ma{3}{2}$.  We claim that any such tour that begins at $(0,0)$ must end at a point whose first coordinate is 0.  To see this, note that the vertices that are adjacent to $(1,a)$ are $(0,a\pm 2)$ and $(2,a\pm 2)$. Further, $\phi_M(0,a-2) = \phi_M(0,a+2)$ and $\phi_M(2,a-2) = \phi_M(2,a+2)$, so the two-edge subpath of a lift of a knight's tour in $\Ma{3}{2}$ containing $(1,a)$ in its interior must adjoin vertices with first coordinates 0 and 2. Therefore, the lift can be partitioned into four edge paths of length 1 or 2, each of which adjoins vertices with first coordinates 0 and 2.  Since we begin our edge path at a vertex with first coordinate 0, we must also end at a vertex with first coordinate 0.  Hence, our lift does not map to a generating tour, so there are no generating tours on $\Ma{3}{2}$. 

When $m$ is even, $\Ma{m}{2}$ has no generating tour by Proposition \ref{prop:oddeven}.  Figure \ref{fig:MSGenmx2} shows extendable generating tours for $\Ma{5}{2}$ and $\Ma{7}{2}$.  Hence, there is a generating tour on $\Ma{m}{2}$ when $m$ is odd and $m \geq 5$.

\begin{figure}[t]
\centering
\begin{subfigure}[b]{0.25\textwidth}
\fbox{\includegraphics[width = .85\textwidth]{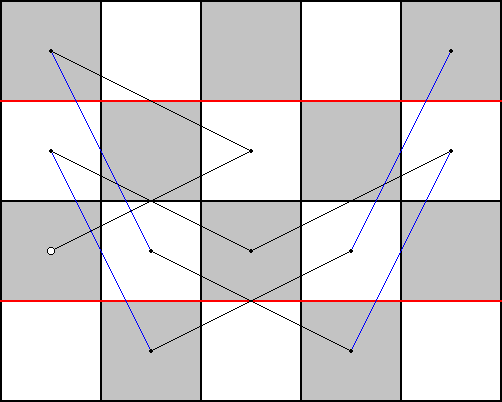}}
\caption{\,}
\end{subfigure}
\qquad
\begin{subfigure}[b]{0.34\textwidth}
\fbox{\includegraphics[width = .85\textwidth]{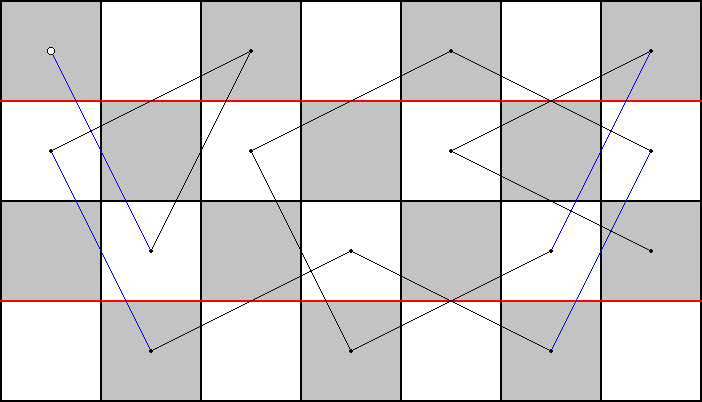}}
\caption{\;}
\end{subfigure}
\caption{Lifts of generating tours on $\Ma{5}{2}$ and $\Ma{7}{2}$.}
\label{fig:MSGenmx2}
\end{figure}

For the remainder of this section, we assume $m$ and $n$ are greater than $2$.

\begin{figure}[t]

\begin{picture}(300,270)
\put(38,140){\begin{subfigure}[b]{0.195\textwidth}
\fbox{\includegraphics[width = .85\textwidth]{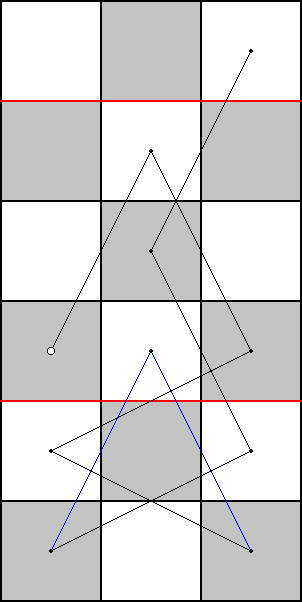}}
\caption{\,}
\end{subfigure}}
\put(128,140){\begin{subfigure}[b]{0.165\textwidth}
\fbox{\includegraphics[width = .85\textwidth]{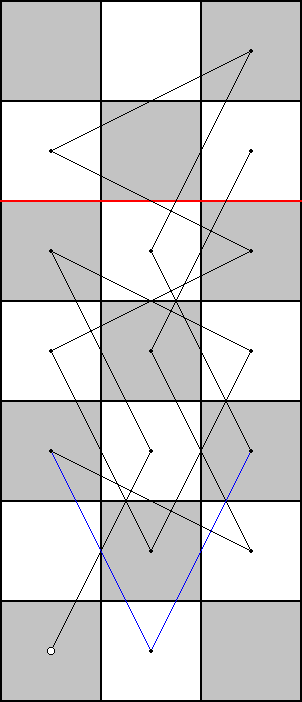}}
\caption{}
\end{subfigure}}
\put(210,140){\begin{subfigure}[b]{0.145\textwidth}
\fbox{\includegraphics[width = .85\textwidth]{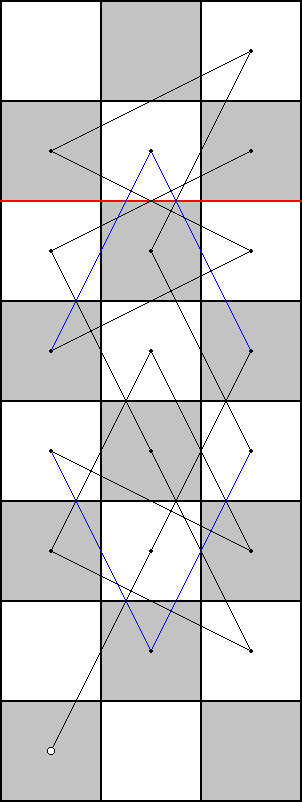}}
\caption{}
\end{subfigure}}

\put(10,0){\begin{subfigure}[b]{0.13\textwidth}
\fbox{\includegraphics[width = .85\textwidth]{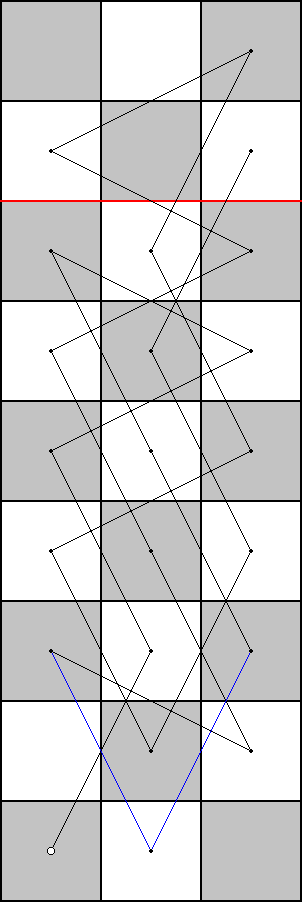}}
\caption{}
\end{subfigure}}
\put(77,0){\begin{subfigure}[b]{0.12\textwidth}
\fbox{\includegraphics[width = .83\textwidth]{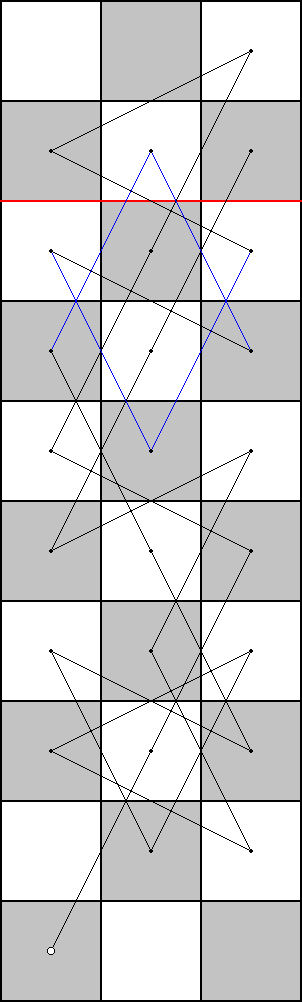}}
\caption{}
\end{subfigure}}
\put(144,0){\begin{subfigure}[b]{0.195\textwidth}
\fbox{\includegraphics[width = .85\textwidth]{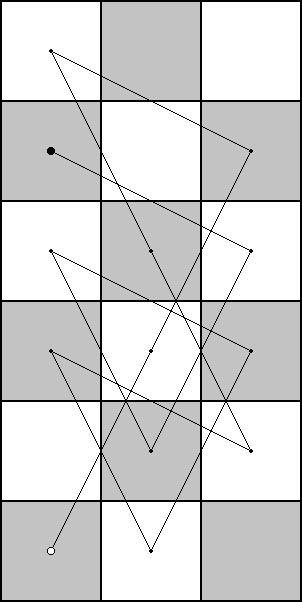}}
\caption{\,}
\end{subfigure}}

\put(230,0){\begin{subfigure}[b]{0.195\textwidth}
\fbox{\includegraphics[width = .85\textwidth]{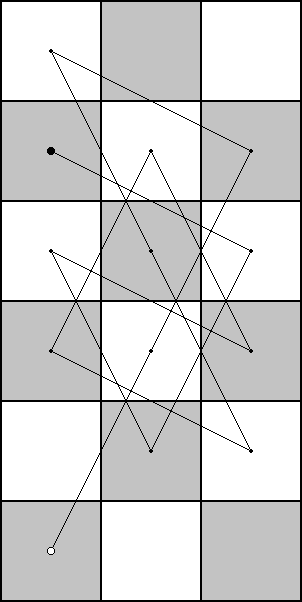}}
\caption{}
\end{subfigure}}

\end{picture}

\caption{Frames (A) to (E): Lifts of generating tours on $\Ma{3}{3}$, $\Ma{3}{5}$, $\Ma{3}{6}$, $\Ma{3}{7}$ and $\Ma{3}{8}$. Frames (F) and (G): Paths in $\Sc{3}$ used in the induction argument for the $3 \times n$ case above.  }
\label{fig:MSGen3xn}
\end{figure}

\subsection*{{\bf $3 \times n$}} By Theorem \ref{thm:watkinsms}, there is no generating tour on $\Ma{3}{4}$.  Frames (A) to (E) of Figure \ref{fig:MSGen3xn} show extendable generating tours on $\Ma{3}{n}$ for $n = 3, 5, 6, 7,$ and $8$.  For $n > 8$, we use induction to produce a generating tour on $\Ma{3}{n}$.  Our induction argument splits into two cases based on the parity of $n$ and uses the graphs in frames (F) and (G) of Figure \ref{fig:MSGen3xn}.  For $n > 8$, we inductively construct a generating tour on $\Ma{3}{n+4}$ by concatenating the generating tour on $\Ma{3}{n}$ with the one of the paths shown in frames (F) and (G) of Figure \ref{fig:MSGen3xn}.  When $n$ is odd, we concatenate by placing the large white vertex in frame (F) of Figure \ref{fig:MSGen3xn} at $(0,0)$, placing the large black vertex at $(0,4)$, and gluing this vertex to the path in $\Sc{3}$ that maps to a generating tour in $\Ma{3}{n-4}$ translated upward by 4. Using the paths in frames (B) and (D) of Figure \ref{fig:MSGen3xn} as base cases, this procedure inductively creates an extendable generating tour on $\Ma{3}{n}$ when $n$ is odd and $n \geq 5$. When $n$ is even, we concatenate by placing the large white vertex in frame (G) of Figure \ref{fig:MSGen3xn} at $(0,0)$, placing the large black vertex at $(0,4)$, and gluing this vertex to the path in $\Sc{3}$ that maps to a generating tour in $\Ma{3}{n-4}$ translated upward by 4. Using the paths in frames (C) and (E) of Figure \ref{fig:MSGen3xn} as base cases, this procedure inductively creates an extendable generating tour on $\Ma{3}{n}$ when $n$ is even and $n \geq 6$. Hence, when $n \geq 3$ and $n \ne 4$, there is an extendable generating tour on $\Ma{3}{n}$.

\subsection*{{\bf $m \times 3$}} By Theorem \ref{thm:watkinsms}, there is no generating tour on $\Ma{4}{3}$. Figure \ref{fig:MSGenmx3} shows extendable generating tours on $\Ma{m}{3}$ for $m = 5, 6,$ and $8$. Hence, when $m \geq 3$ and $m \ne 4$, there is a generating tour on $\Ma{m}{3}$.

For the remainder of this section, we assume $m$ and $n$ are greater than $3$.

\begin{figure}[t]
\centering
\begin{subfigure}[b]{0.2\textwidth}
\fbox{\includegraphics[width = .85\textwidth]{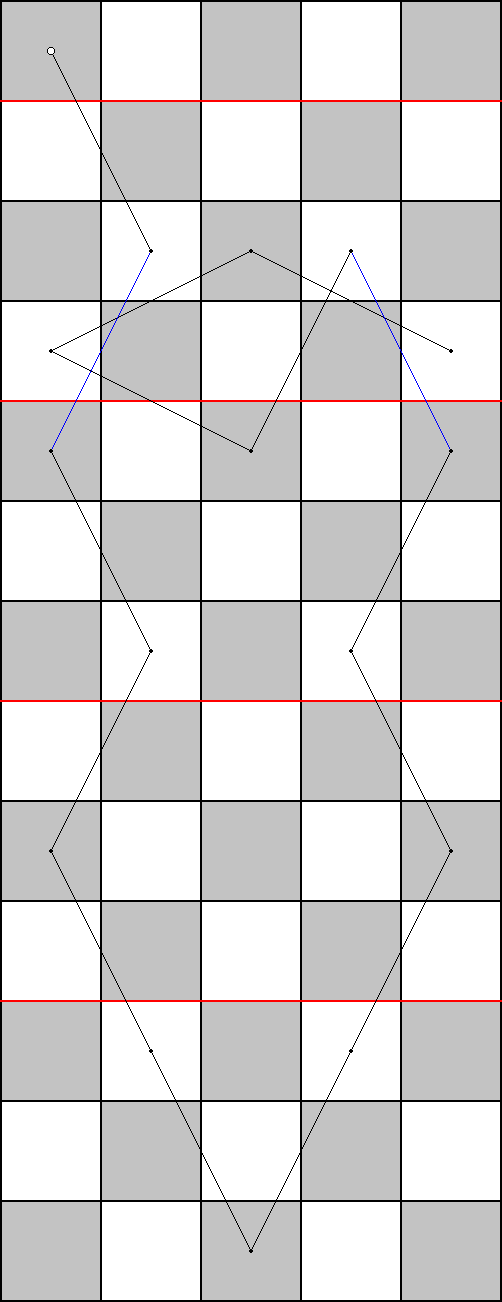}}
\caption{}
\end{subfigure}
\qquad 
\begin{subfigure}[b]{0.22\textwidth}
\fbox{\includegraphics[width = .85\textwidth]{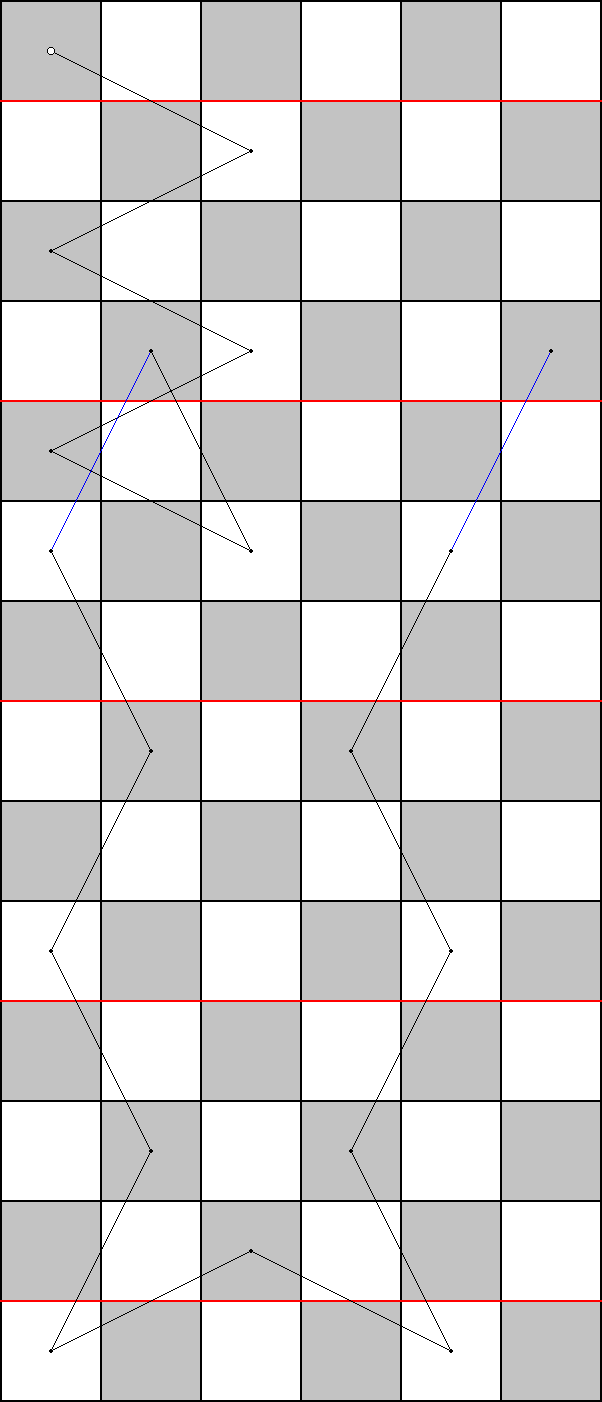}}
\caption{\,}
\end{subfigure}
\qquad
\begin{subfigure}[b]{0.34\textwidth}
\fbox{\includegraphics[width = .85\textwidth]{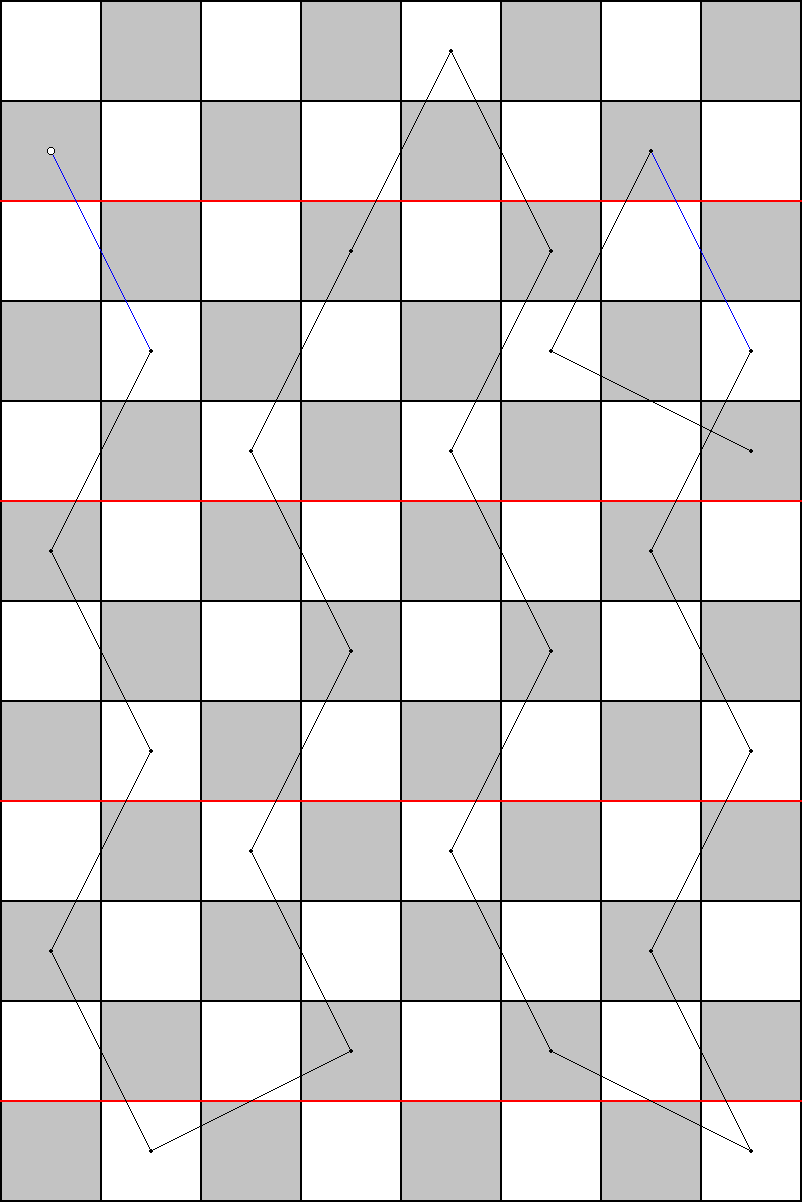}}
\caption{\;\:}
\end{subfigure}

\caption{Lifts of generating tours on $\Ma{5}{3}$, $\Ma{6}{3}$, and $\Ma{8}{3}$.}
\label{fig:MSGenmx3}
\end{figure}

\subsection*{{\bf $4 \times n$}} By Theorem \ref{thm:watkinsms}, there is no generating tour on $\Ma{4}{n}$ when $n$ is odd.  By Proposition \ref{prop:oddeven}, there is no generating tour on $\Ma{4}{n}$ when $n$ is even.

\subsection*{{\bf $m \times 4$}} By Proposition \ref{prop:oddeven}, there is no generating tour on $\Ma{m}{4}$ when $m$ is even.  Figure \ref{fig:MSGenmx4} shows extendable generating tours on $\Ma{5}{4}$ and $\Ma{7}{4}$.  Hence, $\Ma{m}{4}$ has a generating tour when $m$ is odd and $m \geq 5$.

\begin{figure}[t]

\centering
\begin{subfigure}[b]{0.25\textwidth}
\fbox{\includegraphics[width = .85\textwidth]{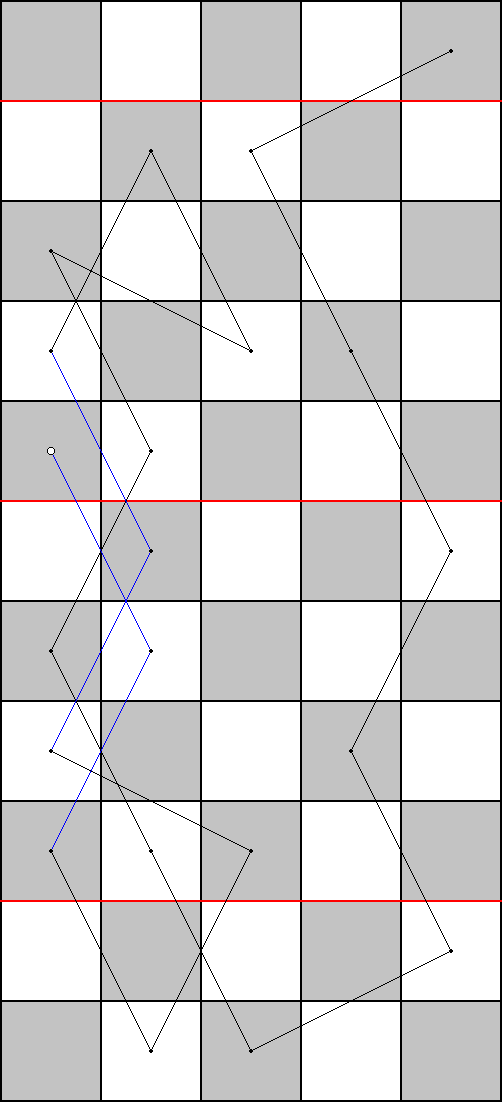}}
\caption{\,}
\end{subfigure}
\qquad \qquad
\begin{subfigure}[b]{0.39\textwidth}
\fbox{\includegraphics[width = .845\textwidth]{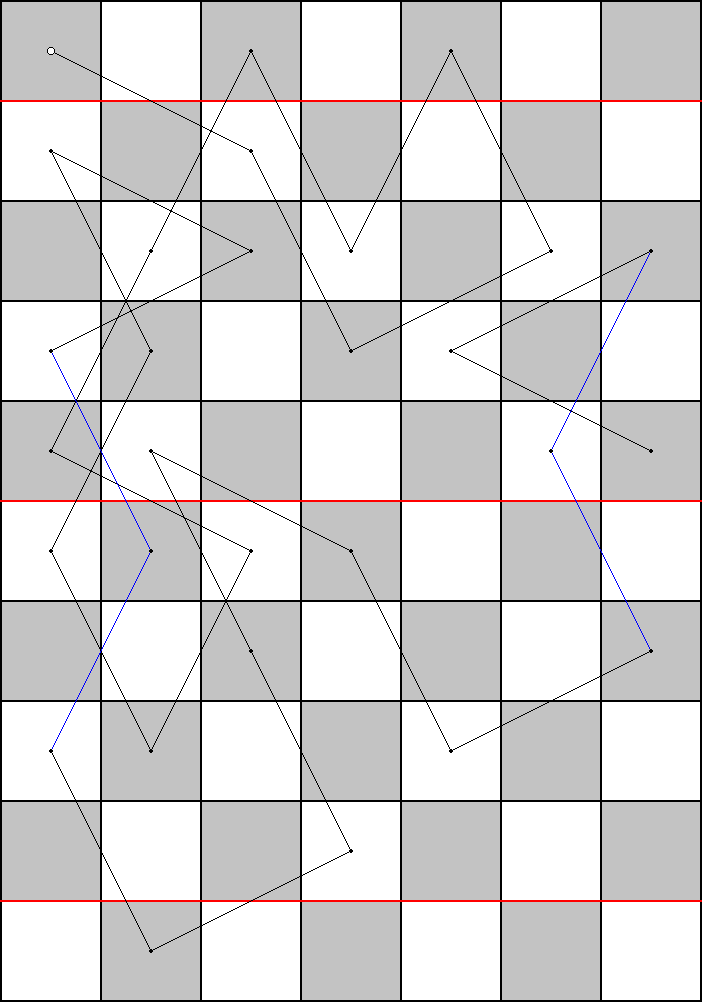}}
\caption{\;\;\;}
\end{subfigure}
\caption{Lifts of generating tours on $\Ma{5}{4}$ and $\Ma{7}{4}$.}
\label{fig:MSGenmx4}
\end{figure}

\begin{figure}[h]
\centering
\begin{subfigure}[b]{0.12\textwidth}
\fbox{\includegraphics[width = .85\textwidth]{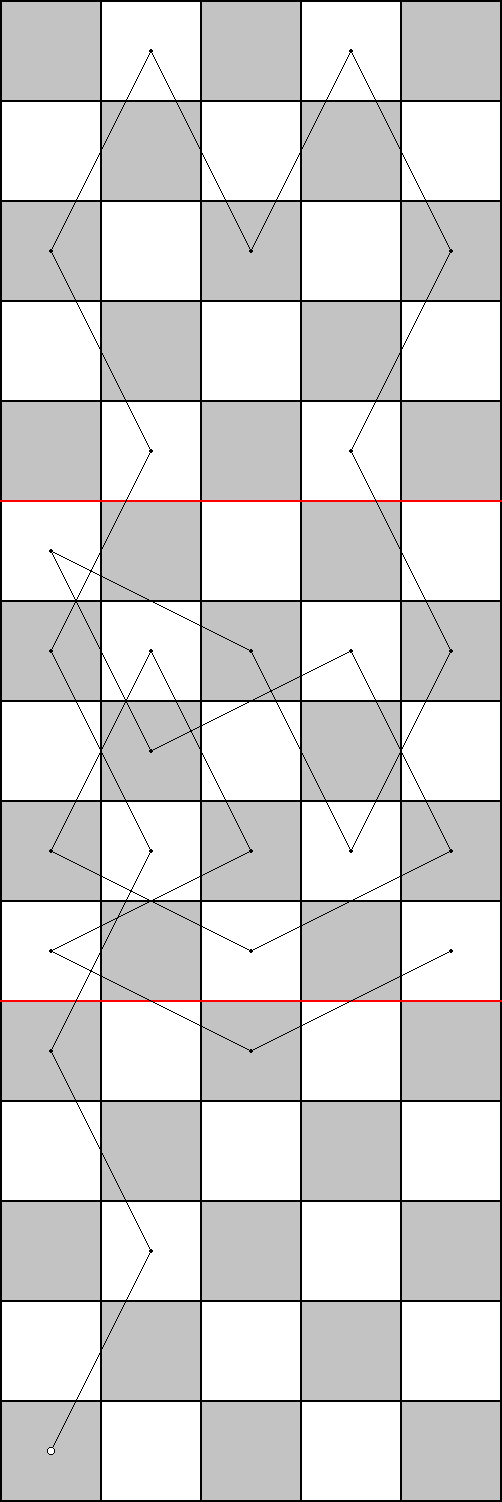}}
\caption{}
\end{subfigure}
\hfill
\begin{subfigure}[b]{0.225\textwidth}
\fbox{\includegraphics[width = .85\textwidth]{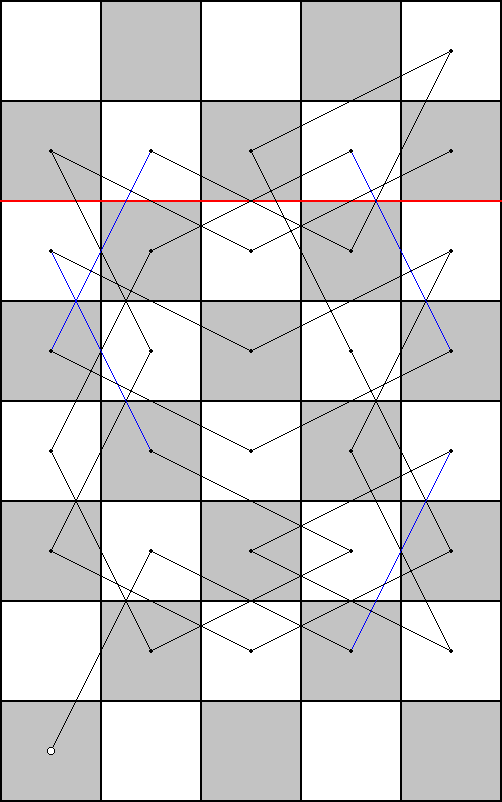}}
\caption{\,}
\end{subfigure}
\hfill
\begin{subfigure}[b]{0.18\textwidth}
\fbox{\includegraphics[width = .85\textwidth]{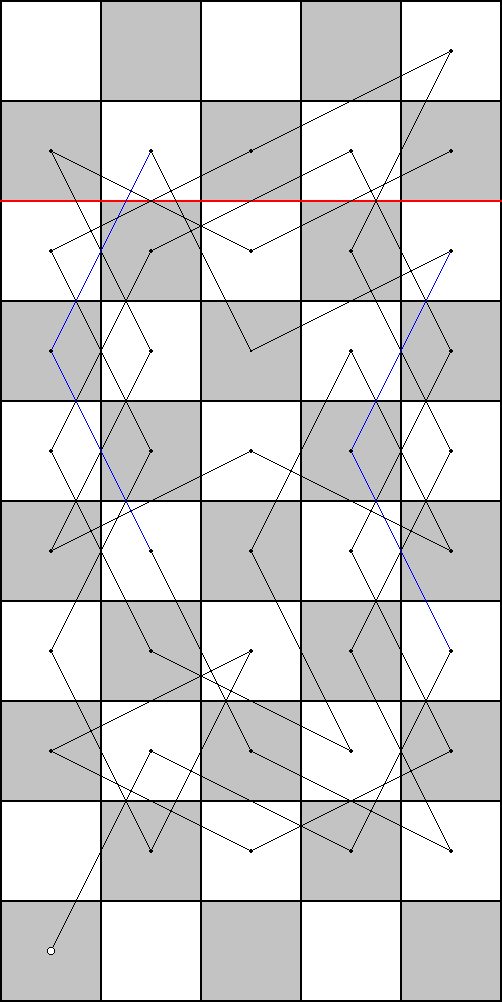}}
\caption{}
\end{subfigure}
\hfill
\begin{subfigure}[b]{0.305\textwidth}
\fbox{\includegraphics[width = .84\textwidth]{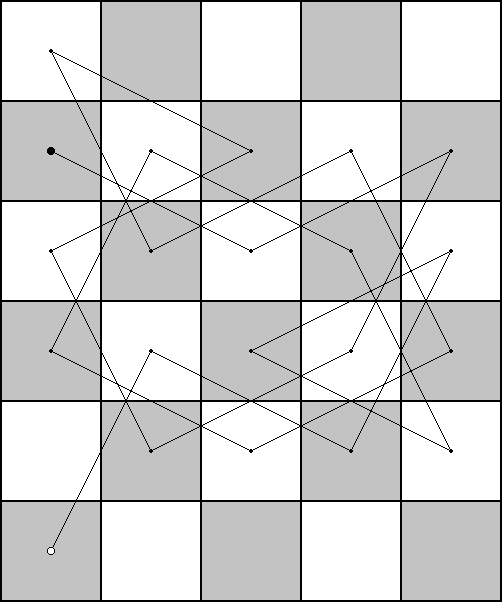}}
\caption{\;\:}
\end{subfigure}
\caption{Frames (A), (B), and (C) are lifts of generating tours in $\Ma{5}{5}$, $\Ma{5}{6}$, and $\Ma{5}{8}$.  Frame (D) is a path in $\Sc{5}$ used in the induction argument for the $5 \times n$ case above.}
\label{fig:MSGen5xn}
\end{figure}

For the remainder of this section, we assume $m$ and $n$ are greater than $4$.

\subsection*{{\bf $5 \times n$}} Frame (A) of Figure \ref{fig:MSGen5xn} shows a generating tour on $\Ma{5}{5}$ and frames (B) and (C) of Figure \ref{fig:MSGen5xn} show extendable generating tours on $\Ma{5}{6}$ and $\Ma{5}{8}$.  When $n$ is even and $n > 8$, we construct tours by concatenating the path in $\Sc{5}$ whose image is a generating tour in $\Ma{5}{n-4}$ translated upward by 4 with the image of the path shown in frame (D) of Figure \ref{fig:MSGen5xn}.  We take the white vertex in frame (D) as $(0,0)$ and notice that the terminal point of the path in frame (D) at $(0,4)$ aligns with the other path. Taking frames (B) and (C) of Figure \ref{fig:MSGen5xn} as base cases, this procedure inductively creates an extendable generating tour on $\Ma{5}{n}$ when $n$ is even and $n \geq 6$.

For the case when $n$ is odd and $n \geq 7$, note that by Theorem \ref{thm:ralston}, there is an open knight's tour on $\Ma{5}{n}$ that begins at $(0,0)$ and ends at $(2,n-1)$.  Making one additional move to $(4,n) = (0,0)$ creates a generating knight's tour on $\Ma{5}{n}$.  Hence, there is a generating tour on $\Ma{5}{n}$ for all $n \geq 5$.

At this point, we will no longer increment $n$ and only assume that $n \geq 5$.  Due to the similarity in the argument, we consider the $6 \times n$ and $8 \times n$ simultaneously.

\begin{figure}[t]

\centering
\begin{subfigure}[b]{0.34\textwidth}
\fbox{\includegraphics[width = .85\textwidth]{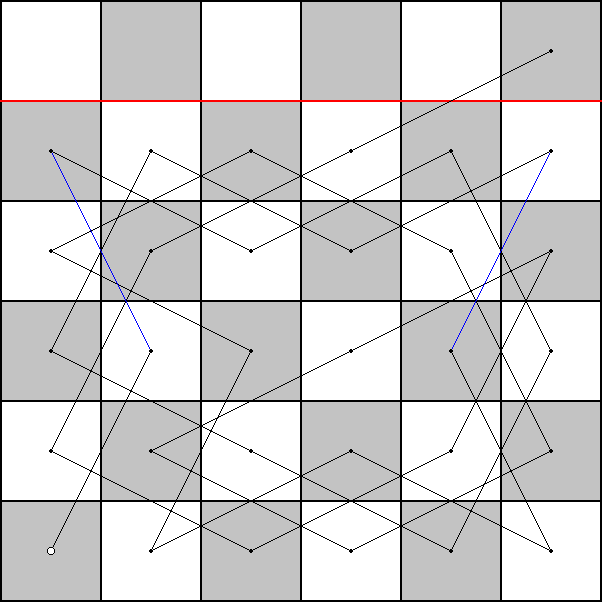}}
\caption{\;\:}
\end{subfigure}
\qquad
\begin{subfigure}[b]{0.255 \textwidth}
\fbox{\includegraphics[width = .85\textwidth]{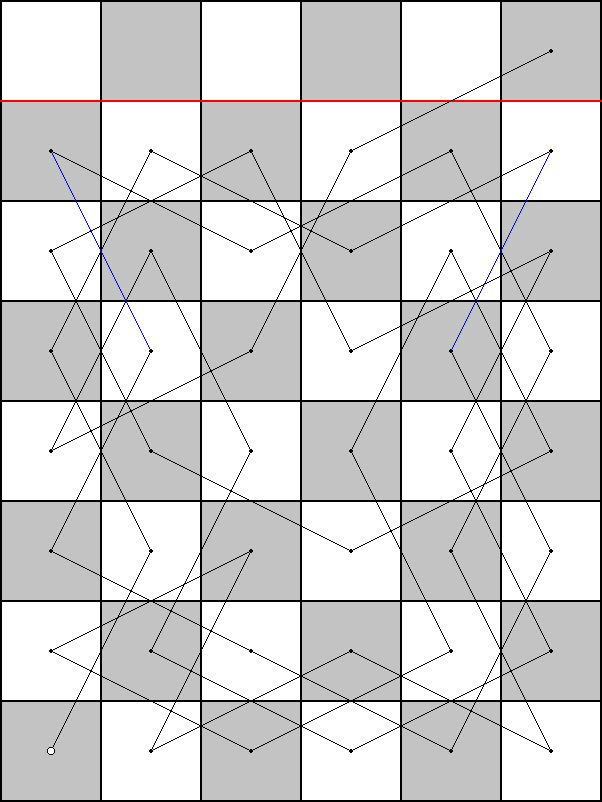}}
\caption{\,}
\end{subfigure}\\[8 pt]
\begin{subfigure}[b]{0.37\textwidth}
\fbox{\includegraphics[width = .85\textwidth]{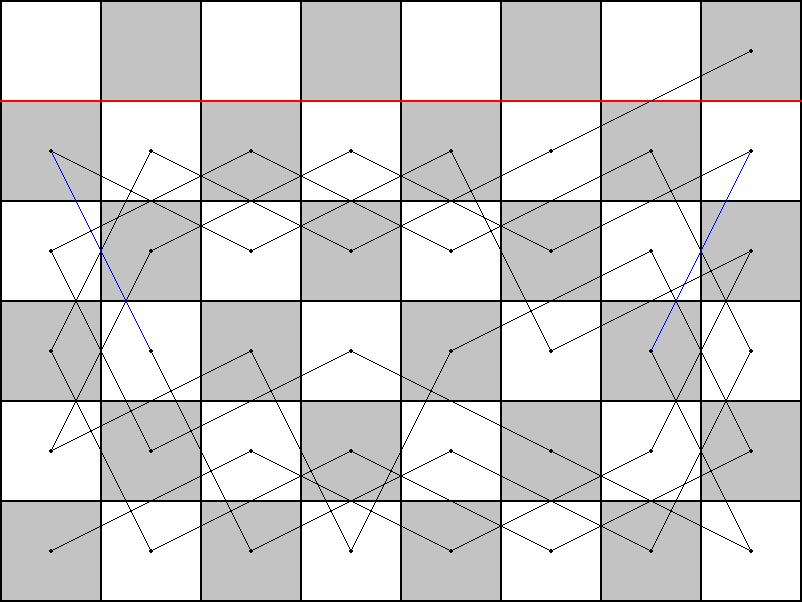}}
\caption{\;\;}
\end{subfigure}
\qquad
\begin{subfigure}[b]{0.28\textwidth}
\fbox{\includegraphics[width = .85\textwidth]{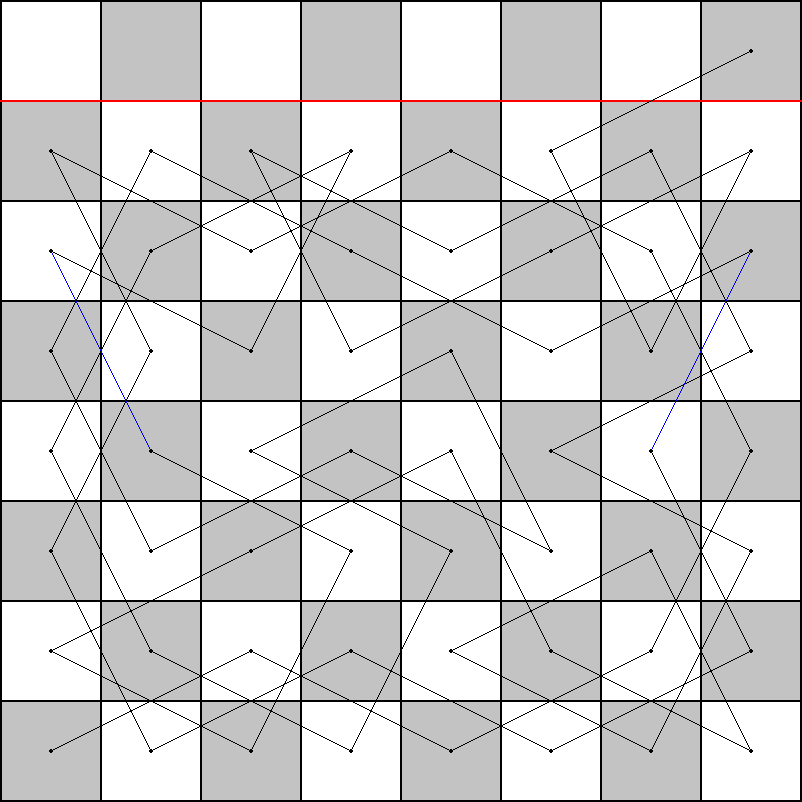}}
\caption{\;}
\end{subfigure}
\caption{Lifts of generating tours on $\Ma{6}{5}$, $\Ma{6}{7}$ $\Ma{8}{5}$, and $\Ma{8}{7}$.}
\label{fig:MSGen68xn}
\end{figure}

\subsection*{$6 \times n$, $8 \times n$} By Proposition \ref{prop:oddeven}, there is no generating tour on $\Ma{m}{n}$ when $m = 6$ or $8$ and $n$ is even.  Figure \ref{fig:MSGen68xn} shows extendable generating tours on $\Ma{6}{5}$, $\Ma{6}{7}$, $\Ma{8}{5}$, and $\Ma{8}{7}$. For $n$ odd with $n > 7$, we construct generating tours on $\Ma{m}{n}$ via induction. We translate the path on $\Sc{m}$ that gives a generating tour on $\Ma{m}{n-4}$ upward by $4$ and concatenate this translated path with frame (A) or frame (B) of Figure \ref{fig:MSGen68xnInductiion} when $m=6$ or $m=8$ respectively. In particular, we take the white vertices in frames (A) and (B) of Figure \ref{fig:MSGen68xnInductiion} as $(0,0)$.  This produces an extendable generating tour on $\Ma{m}{n}$ when $m = 6$ or $8$ and $n$ is odd with $n \geq 5$.

\begin{figure}[t]

\centering
\begin{subfigure}[b]{0.25\textwidth}
\fbox{\includegraphics[width = .85\textwidth]{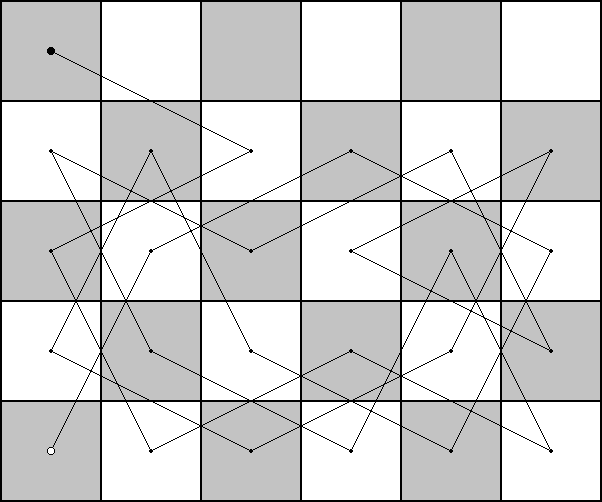}}
\caption{\:}
\end{subfigure}
\qquad
\begin{subfigure}[b]{0.34 \textwidth}
\fbox{\includegraphics[width = .85\textwidth]{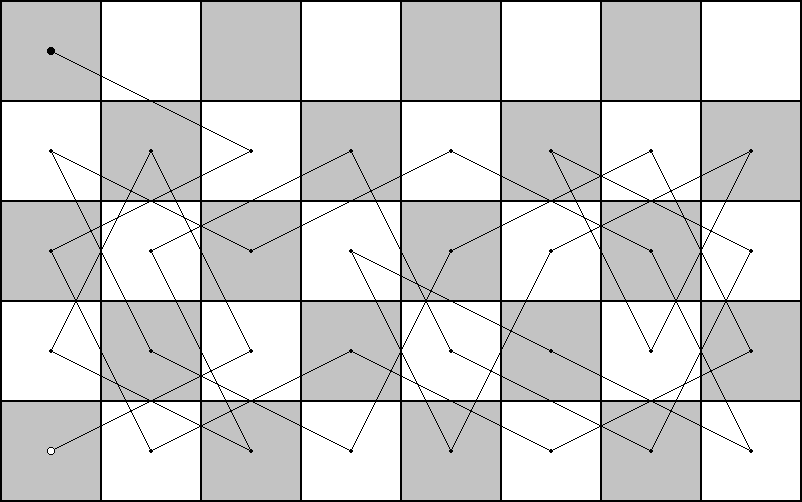}}
\caption{\;\:}
\end{subfigure}

\caption{Paths in $\Sc{6}$ and $\Sc{8}$ used in the induction argument for the $6 \times n$, $8 \times n$ case above.}
\label{fig:MSGen68xnInductiion}
\end{figure}

\subsection*{{\bf $m \times n$} when $m \geq 7$ and $m$ is odd}  When $n$ is even, there is a generating tour given by applying Corollary \ref{cor:widen} to the tours constructed inductively for $\Ma{3}{n}$ and $\Ma{5}{n}$.  When $n$ is odd, there exists an open knight's tour from $(0,0)$ to $(m-3,n-1)$ on $\Ma{m}{n}$ by Theorem \ref{thm:ralston}.  Concatenating this with the edge from $(m-3,n-1)$ to $(m-1,n) = (0,0)$ gives a generating tour on $\Ma{m}{n}.$

\subsection*{{\bf $m \times n$} when $m \geq 10$ and $m$ is even}  By Proposition \ref{prop:oddeven}, there is no generating tour on $\Ma{m}{n}$ when $n$ is even.  When $n$ is odd, there is a generating tour on $\Ma{m}{n}$ given by Corollary \ref{cor:widen} and the tours constructed inductively for $\Ma{6}{n}$ and $\Ma{8}{n}$.

This completes the proof of Theorem \ref{thm:MSGen}.

\section{\bf Nullhomotopic, Cylindrical, and M\"{o}bius Tours on Klein Bottles}

We now turn to Klein bottles, extending the results of the previous two sections. In this section, we classify the dimensions of Klein bottle boards that admit a nullhomotopic tour, the dimensions that admit a cylindrical tour, and the dimensions that admit a M\"{o}bius tour.  These classifications are given in Theorems \ref{thm:KBNull}, \ref{thm:KBCyl}, and \ref{thm:KBMob} respectively.

\begin{theorem}
The pseudograph $\Ka{m}{n}$ has a nullhomotopic knight's tour if and only if neither of the following are true:
\begin{itemize}
    \item both $m$ and $n$ are odd and at least one of $m$ and $n$ is greater than 1; or
    \item both $m$ and $n$ are less than or equal to 2 and at least one of $m$ and $n$ is greater than 1.
\end{itemize}
\label{thm:KBNull}
\end{theorem}

When considering tours on Klein bottles, we can draw on results from both M\"{o}bius strips and cylinders.  Recall that $\Ka{m}{n}$ canonically contains $\Ma{m}{n}$ and so we can apply Theorem \ref{thm:MSNH}.  Additionally, $\Ka{m}{n}$ canonically contains $\Ca{m}{n}$, enabling the application of Theorem \ref{thm:forrest-teehan}.  Together with Proposition \ref{prop:oddeven}, these theorems prove Theorem \ref{thm:KBNull} except for five specific cases.  Figure \ref{fig:KBNull} shows nullhomotopic tours on $\Ka{2}{4}$ and $\Ka{4}{1}$, leaving only the case when both $m$ and $n$ are less than or equal to 2 and at least one of $m$ and $n$ is greater than 1.

\begin{figure}[t]

\centering
\begin{subfigure}[b]{0.25\textwidth}
\fbox{\includegraphics[width = .85\textwidth]{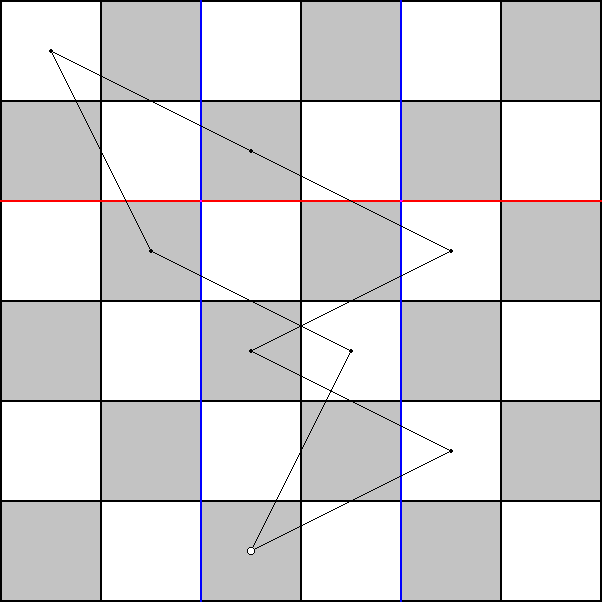}}
\caption{\,}
\end{subfigure}
\qquad
\begin{subfigure}[b]{0.31 \textwidth}
\fbox{\includegraphics[width = .85\textwidth]{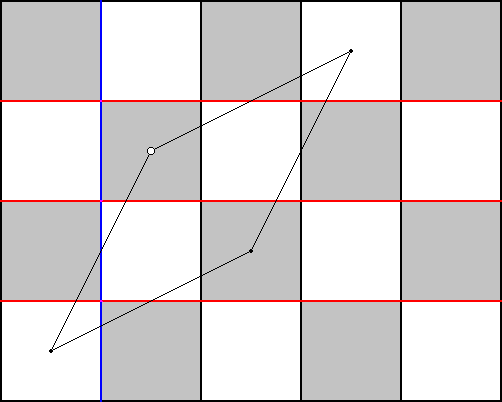}}
\caption{\;\,}
\end{subfigure}
\caption{Lifts of nullhomotopic tours on $\Ka{2}{4}$ and $\Ka{4}{1}$.}
\label{fig:KBNull}
\end{figure}

Every edge cycle in $\Pc$ has at least 4 edges.  Hence, there are no nullhomotopic tours on $\Ka{1}{2}$ or $\Ka{2}{1}$.  Each vertex in $\Pc$ is contained in 24 length 4 edge cycles.  None of those 24 edge cycles containing $(0,0)$ in $\Pc$ map to a knight's tour in $\Ka{2}{2}$, so there is no nullhomotopic tour on $\Ka{2}{2}$.  This completes the proof of Theorem \ref{thm:KBNull}.

\begin{theorem}
The pseudograph $\Ka{m}{n}$ with at least one of $m$ and $n$ greater than one has a cylindrical tour if and only if none of the following are true:
\begin{itemize}
\item $m = 2$ and $n = 2$; or
\item $m$ is odd and $n$ is even.
\end{itemize}
\label{thm:KBCyl}
\end{theorem}

The majority of the cases in Theorem \ref{thm:KBCyl} have been established by Proposition \ref{prop:oddeven} and by Theorem \ref{thm:forrest-teehan2}.  All that remains to prove Theorem \ref{thm:KBCyl} is to consider the cases when $n = 1$ or when $n = 2$ or $4$ and $m$ is even. 

\subsection*{{\bf $m \times 1$}} Frames (A), (B), and (C) of Figure \ref{fig:KBCGenmx1} show paths on $\Pc$ that map via $\phi_K$ to cylindrical tours on $\Ka{2}{1}$, $\Ka{3}{1}$, and $\Ka{5}{1}$ respectively.  When $m$ is even, we build a path in $\Pc$ that maps to a cylindrical tour on $\Ka{m}{1}$ by repeating the path shown in Frame (A) $m/2$ times with one copy of this path beginning at $(k,0)$ for each even integer $k$ with $0 \leq k < m$.  When $m$ is odd, we build a path inductively by cutting the path on $\Pc$ that maps to $\Ka{m-2}{1}$ at $(1,1)$ and shifting the part of the path occurring after $(1,1)$ two squares to the right.  To connect these, we concatenate each part with the path shown in frame (A) where the white point is placed at $(1,1)$.  When this procedure is applied to the path shown in frame (B) of Figure \ref{fig:KBCGenmx1}, we obtain the path shown in frame (C) of that figure.

\begin{figure}[t]
\centering
\begin{subfigure}[b]{0.15\textwidth}
\fbox{\includegraphics[width = .82\textwidth]{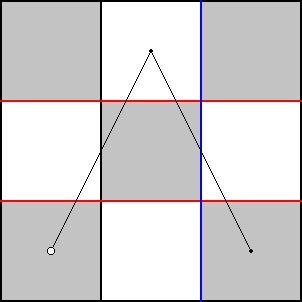}}
\caption{\,}
\end{subfigure}
\qquad 
\begin{subfigure}[b]{0.24\textwidth}
\fbox{\includegraphics[width = .85\textwidth]{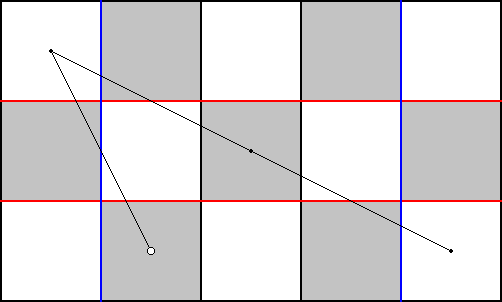}}
\caption{\,}
\end{subfigure}
\qquad
\begin{subfigure}[b]{0.255\textwidth}
\fbox{\includegraphics[width = .85\textwidth]{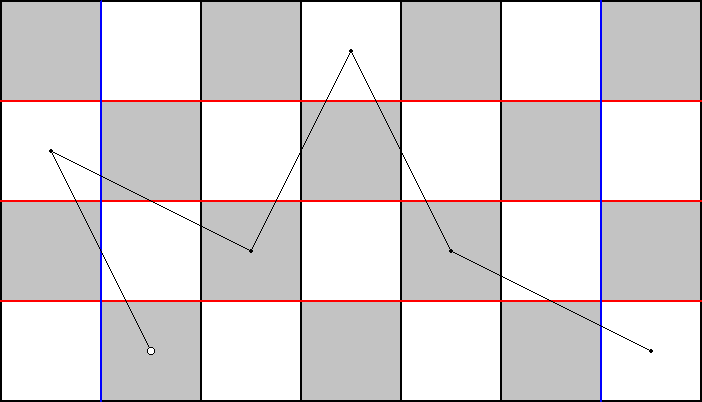}}
\caption{\;\;\;}
\end{subfigure}
\caption{Lifts of cylindrical tours on $\Ka{2}{1}$, $\Ka{3}{1}$, and $\Ka{5}{1}$. The path in frame (A) is used in the induction argument in the $m \times 1$ subsection above.}
\label{fig:KBCGenmx1}
\end{figure}

\begin{figure}[t]

\centering
\begin{subfigure}[b]{0.24 \textwidth}
\fbox{\includegraphics[width = .82\textwidth]{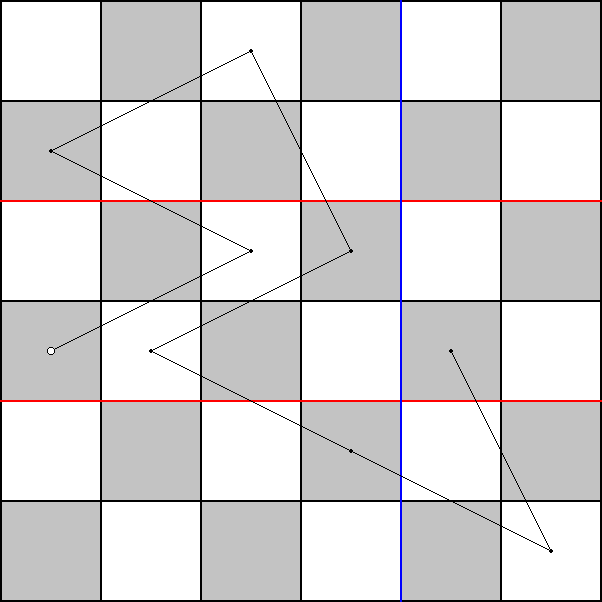}}
\caption{\,}
\end{subfigure}
\quad 
\begin{subfigure}[b]{0.31\textwidth}
\fbox{\includegraphics[width = .85\textwidth]{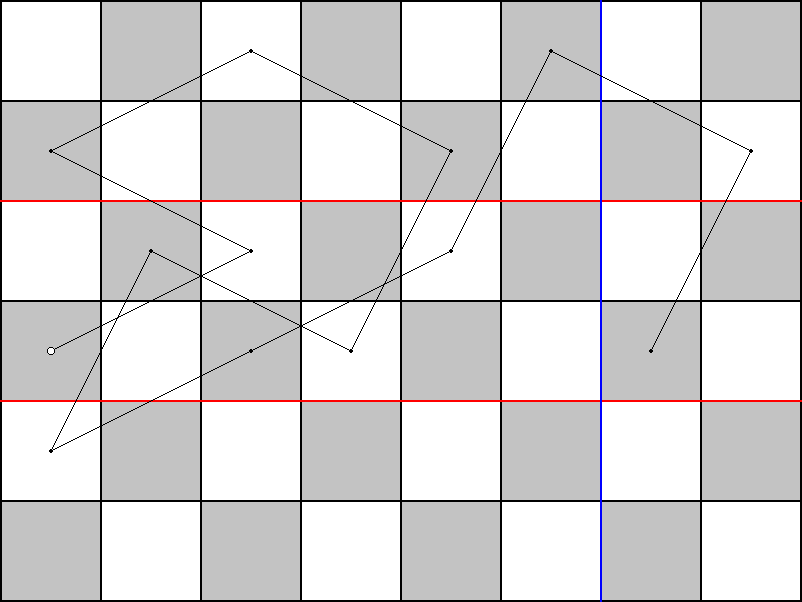}}
\caption{\,}
\end{subfigure}
\quad
\begin{subfigure}[b]{0.31 \textwidth}
\fbox{\includegraphics[width = .85\textwidth]{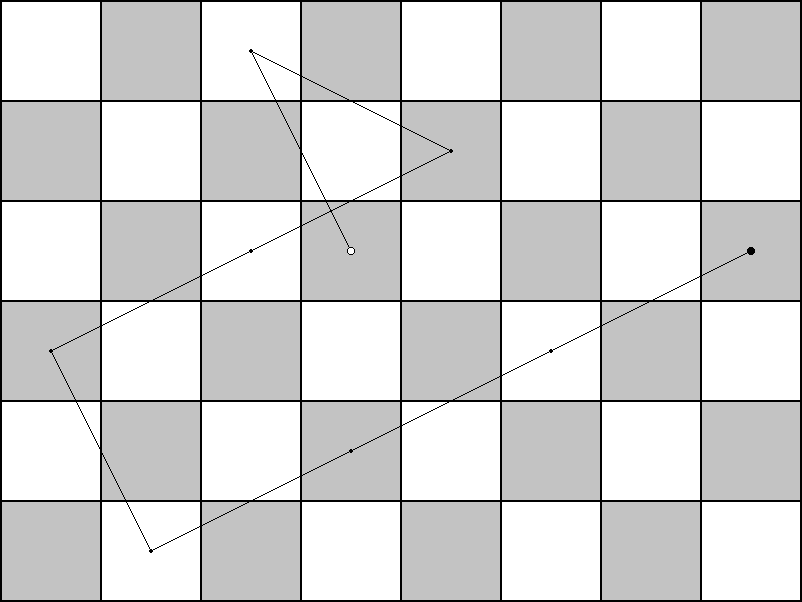}}
\caption{\;\:}
\end{subfigure}

\caption{Lifts of cylindrical tours on $\Ka{4}{2}$ and $\Ka{6}{2}$ are given in frames (A) and (B) respectively. The path in frame (C) is used in the induction argument in the $m \times 2$ subsection above.}
\label{fig:KBCGenmx2}
\end{figure}

\subsection*{{\bf $m \times 2$} when $m$ is even} First consider the case when $m=2$.  Note that $\phi_K(a,b) = \phi_K(a \pm 1, b \pm 2)$ for all vertices $(a,b) \in \Pc$.  Therefore, each edge in a knight's tour on $\Ka{2}{2}$ must lift to an edge in $\Pc$ from $(a,b)$ to $(a \pm 2, b \pm 1)$.  Suppose now that there is a cylindrical tour on $\Ma{2}{2}$ and lift this tour to $\Pc$ so that the initial vertex is $(0,0)$.  By Corollary \ref{cor:cover}, the terminal vertex is $(2,0)$.  Given that four edges are traversed and each edge adds or subtracts 2 to the first coordinate, the possible values of the first coordinate of the terminal vertex of this lift are $0, \pm 4,$ and $\pm 8$, which gives a contradiction. Thus, there is no cylindrical tour on $\Ka{2}{2}$.

Continuing, note that frames (A) and (B) of Figure \ref{fig:KBCGenmx2} show paths on $\Pc$ that map via $\phi_K$ to cylindrical tours on $\Ka{4}{2}$ and $\Ka{6}{2}$ respectively.  When $m > 6$ and $m$ is equivalent to 0 mod 4, we build a path in $\Pc$ that maps to a cylindrical tour on $\Ka{m}{2}$ by repeating the path shown in frame (A) of Figure \ref{fig:KBCGenmx2} $m/4$ times with one copy of this path beginning at $(k,0)$ for each $k$ with $0 \leq k < m$ where $k$ is a multiple of 4. When $m > 6$ and $m$ is equivalent to 2 mod 4, we build a path inductively by cutting the path on $\Pc$ that maps to $\Ka{m-4}{2}$ at $(4,1)$ and shifting the part of the path occurring after $(4,1)$ four squares to the right.  To connect these, we concatenate each part with the path shown in frame (C) of Figure \ref{fig:KBCGenmx2} where the white point is placed at $(4,1)$.

\begin{figure}[t]
\centering
\begin{subfigure}[b]{0.24 \textwidth}
\fbox{\includegraphics[width = .82\textwidth]{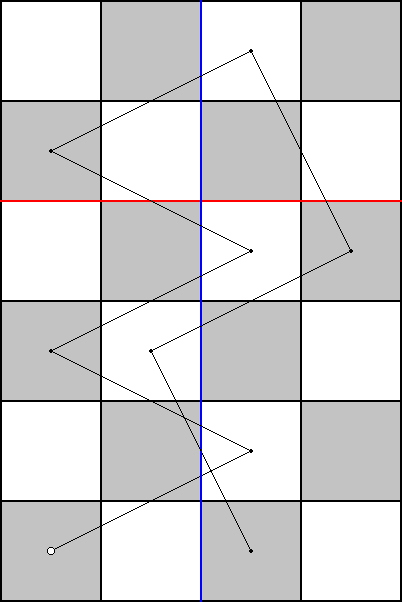}}
\caption{\,}
\end{subfigure}
\caption{Lift of a cylindrical tour on $\Ka{2}{4}$.}
\label{fig:KBCGenmx4}
\end{figure}

\subsection*{{\bf $m \times 4$} when $m$ is even}  Note that Figure \ref{fig:KBCGenmx4} shows a path on $\Pc$ that maps via $\phi_K$ to a cylindrical tour on $\Ka{2}{4}$.  When $m \geq 4$, we repeat this path $m/2$ times, with one copy of this path beginning at $(k,0)$ for each even integer $k$ with $0 \leq k < m$, to build a cylindrical tour on $\Ka{m}{4}$.

This completes the proof of Theorem \ref{thm:KBCyl}.  

\begin{theorem}
The pseudograph $\Ka{m}{n}$ with at least one of $m$ and $n$ greater than one has a M\"{o}bius tour if and only if at least one of $m$ and $n$ is odd. 
\label{thm:KBMob}
\end{theorem}

The majority of the cases in Theorem \ref{thm:KBMob} have been established by Proposition \ref{prop:oddeven} and Theorem \ref{thm:MSGen}. Since $\Ma{m}{n}$ is a subgraph of $\Ka{m}{n}$, we can apply Theorem \ref{thm:MSGen} in this context.  Further, note that $\Ka{1}{n}$ is canonically equal to the pseudograph $\Ta{1}{n}$ from \cite{forrest-teehan}, and, by Theorem 6.1 of \cite{forrest-teehan}, there is a M\"obius tour on $\Ka{1}{n}$ when $n > 1$.

\begin{figure}[t]

\centering
\begin{subfigure}[b]{0.196\textwidth}
\fbox{\includegraphics[width = .85\textwidth]{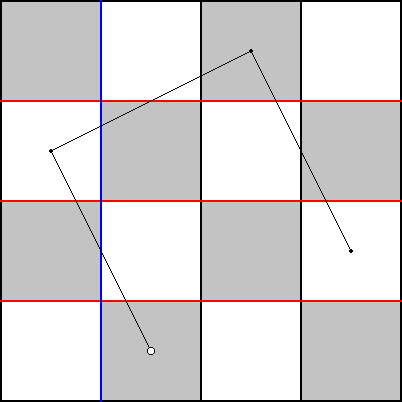}}
\caption{}
\end{subfigure}
\hfill
\begin{subfigure}[b]{0.245\textwidth}
\fbox{\includegraphics[width = .85\textwidth]{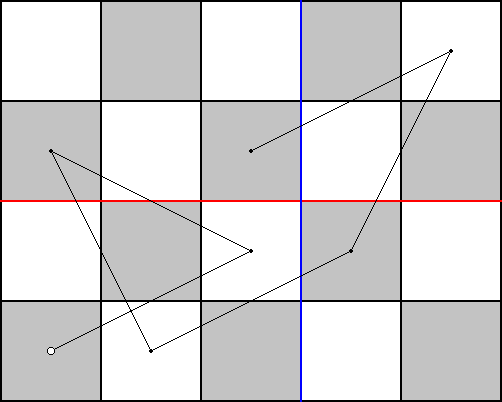}}
\caption{\,}
\end{subfigure}
\hfill
\begin{subfigure}[b]{0.195\textwidth}
\fbox{\includegraphics[width = .85\textwidth]{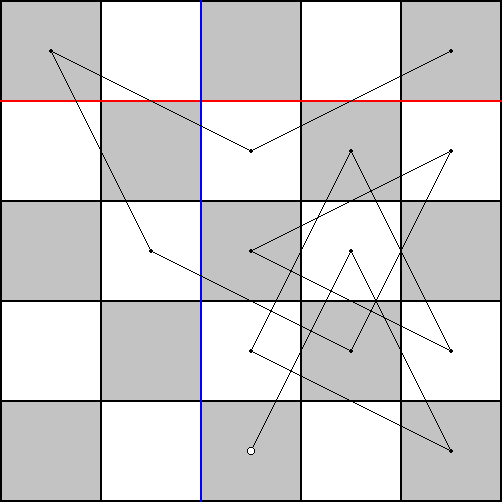}}
\caption{}
\end{subfigure}
\hfill
\begin{subfigure}[b]{0.29\textwidth}
\fbox{\includegraphics[width = .85\textwidth]{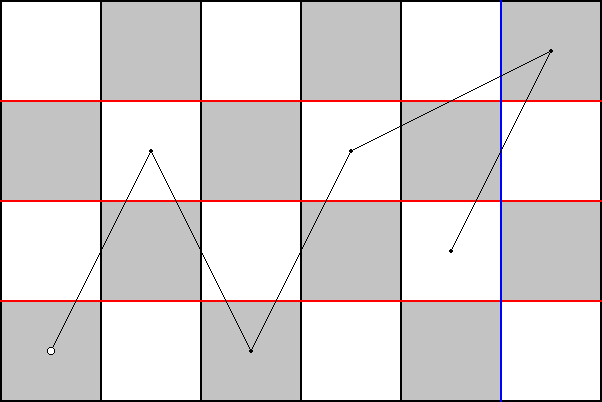}}
\caption{\;}
\end{subfigure}
\caption{Lifts of M\"obius tours on $\Ka{3}{1}$, $\Ka{3}{2}$, $\Ka{3}{4}$, and $\Ka{5}{1}$.}
\label{fig:KBMGenSporadic}
\end{figure}

All that remains in order to prove Theorem \ref{thm:KBMob} is to consider the cases when $m = 2$ or $4$ and $n$ is even along with the sporadic multigraphs $\Ka{3}{1}$, $\Ka{3}{2}$, $\Ka{3}{4}$, and $\Ka{5}{1}$.  M\"obius tours on the four sporadic Klein boards are shown in Figure \ref{fig:KBMGenSporadic}. We consider the other cases below.

\subsection*{{\bf $2 \times n$} when $n$ is odd}

Frames (A) and (B) of Figure \ref{fig:KBMGen2xn} show paths on $\Pc$ that map via $\phi_K$ to M\"obius tours on $\Ka{2}{1}$ and $\Ka{2}{3}$ respectively.  For $n > 3$, we build a tour inductively by concatenating the path on $\Pc$ that maps to a tour on $\Ka{2}{n-2}$ with the path shown in frame (C) of Figure \ref{fig:KBMGen2xn}, where the white vertex in frame (C) has coordinates $(1,n-2)$.  Hence, $\Ka{2}{n}$ admits a M\"obius tour for all odd $n$.

\begin{figure}[t]
\centering
\begin{subfigure}[b]{0.26\textwidth}
\fbox{\includegraphics[width = .85\textwidth]{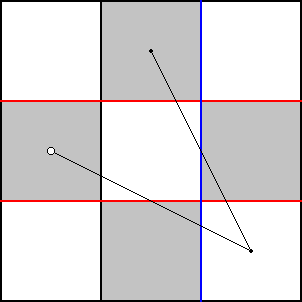}}
\caption{\,}
\end{subfigure}
\qquad
\begin{subfigure}[b]{0.21\textwidth}
\fbox{\includegraphics[width = .85\textwidth]{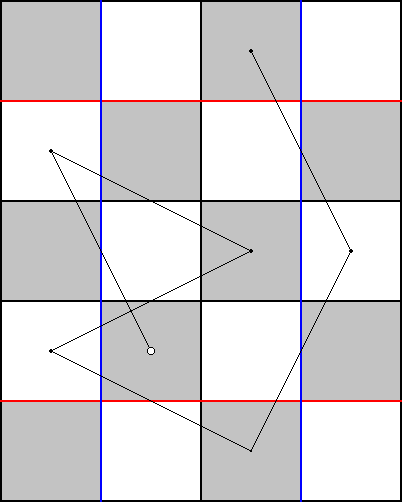}}
\caption{\,}
\end{subfigure}
\qquad
\begin{subfigure}[b]{0.26\textwidth}
\fbox{\includegraphics[width = .85\textwidth]{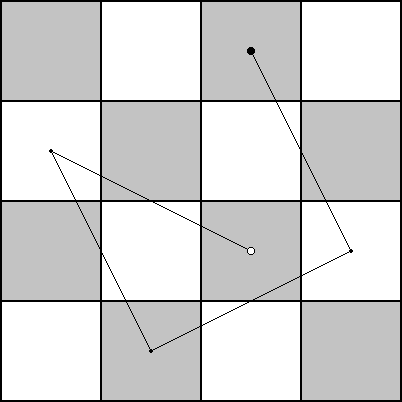}}
\caption{\,}
\end{subfigure}
\caption{Lifts of cylindrical tours on $\Ka{2}{1}$ and $\Ka{2}{3}$ are given in frames (A) and (B) respectively. The path in frame (C) is used in the induction argument in the $2 \times n$ when $n$ is odd subsection above.}
\label{fig:KBMGen2xn}
\end{figure}

\subsection*{{\bf $4 \times n$} when $n$ is odd}

Frames (A), (B), and (C) of Figure \ref{fig:KBMGen4xn} show paths on $\Pc$ that map via $\phi_K$ to M\"obius tours on $\Ka{4}{1}$, $\Ka{4}{3}$, and $\Ka{4}{5}$ respectively. For $n > 5$, we build a path inductively by cutting the path on $\Pc$ that maps to a tour on $\Ka{4}{n-2}$ at $(3,n-2)$ and shifting the part of the path occurring after $(3,n-2)$ two squares upward.  To connect these, we concatenate each part with the path shown in frame (D)  of Figure \ref{fig:KBMGen4xn}, where the white vertex in frame (D) is placed at $(3,n-2)$. Hence, $\Ka{4}{n}$ admits a M\"obius tour for all odd $n$. 

\begin{figure}[t]

\centering
\begin{subfigure}[b]{0.28\textwidth}
\fbox{\includegraphics[width = .85\textwidth]{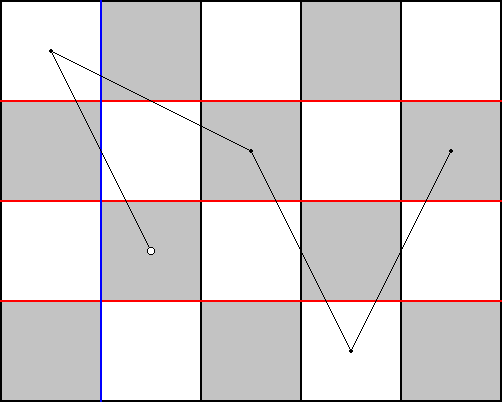}}
\caption{\,}
\end{subfigure}
\hfill
\begin{subfigure}[b]{0.187\textwidth}
\fbox{\includegraphics[width = .85\textwidth]{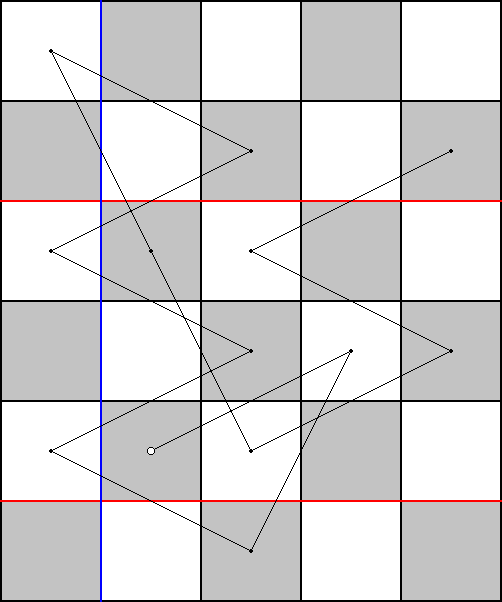}}
\caption{}
\end{subfigure}
\hfill
\begin{subfigure}[b]{0.169\textwidth}
\fbox{\includegraphics[width = .85\textwidth]{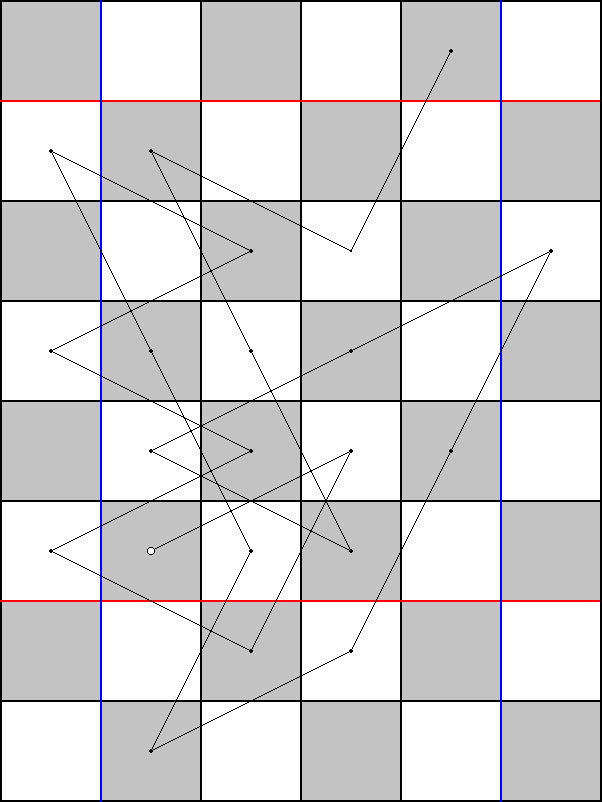}}
\caption{}
\end{subfigure}
\hfill
\begin{subfigure}[b]{0.225\textwidth}
\fbox{\includegraphics[width = .85\textwidth]{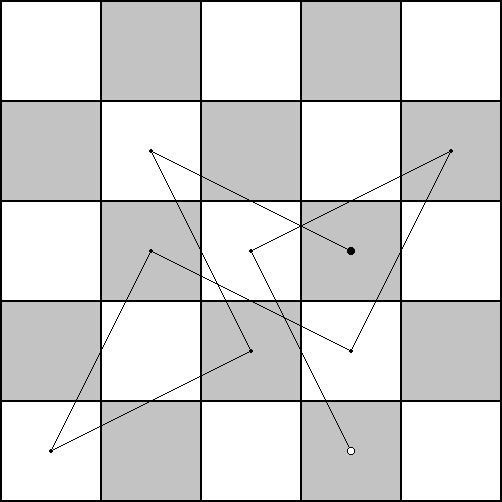}}
\caption{\,}
\end{subfigure}
\caption{Lifts of cylindrical tours on $\Ka{4}{1}$, $\Ka{4}{3}$ and $\Ka{4}{5}$ are given in frames (A), (B), and (C) respectively. The path in frame (D) is used in the induction argument in the $4 \times n$ when $n$ is odd subsection above.}
\label{fig:KBMGen4xn}
\end{figure}

This completes the proof of Theorem \ref{thm:KBMob}.

\newpage
\section{\bf Future Work}

The cylinder, M\"obius strip, torus, and Klein bottle are the compact surfaces that can be obtained as a quotient of an $m \times n$ rectangle that have a homogenous Euclidean metric.  Together with the results in \cite{forrest-teehan}, this work answers for each surface which dimensions of rectangle admit knight's tours that are nullhomotopic and which dimensions admit knight's tours that are homotopic to the standard loop in each direction. This is the initial step in determining, given a surface, which elements of the fundamental group of the surface can be realized by knight's tours.  For example, the fundamental group of both the cylinder and M\"{o}bius strip is isomorphic to $\mathbb{Z}$; what is the largest value in these fundamental groups that can be realized as a knight's tour?  Aside from the questions answered here and in \cite{forrest-teehan}, nothing is known regarding the interaction between the fundamental groups of surfaces and knight's tours.

\bigskip
\newcommand{\doi}[1]{\href{https://doi.org/#1}{Crossref}}
\newcommand{\arXiv}[1]{\href{https://arxiv.org/abs/#1}{arXiv}}
\newcommand{\biblink}[2]{\href{#1}{#2}}
\renewcommand{\MR}[1]{\href{https://mathscinet.ams.org/mathscinet-getitem?mr=#1}{\mbox{MathSciNet}}}

\bibliographystyle{plain}

\end{document}